\documentclass[10pt]{amsart}

\usepackage{amsmath}
\usepackage{amssymb}
\usepackage{amsfonts}%
\usepackage{graphicx}
 
\newtheorem{theorem}{Theorem}[section]
\newtheorem{lemma}{Lemma}[section]
\theoremstyle{definition}

\theoremstyle{remark}
\newtheorem{remark}{Remark}[section]
\numberwithin{equation}{section} 
\theoremstyle{plain}

\newtheorem{proposition}{Proposition}[section]
\newtheorem{example}{Example}[section]

\begin{document}

\title[Return- and hitting-time distributions]{\textbf{Return- and hitting-time distributions of small sets in infinite measure preserving systems}}
\vspace{0.2cm}

\author{Simon Rechberger}
\author{Roland Zweim\"{u}ller}
\address{Fakult\"{a}t f\"{u}r Mathematik, Universit\"{a}t Wien, Oskar-Morgenstern-Platz
1, 1090 Vienna, Austria }
\email{rechberger\_simon@gmx.at ;  roland.zweimueller@univie.ac.at}
\urladdr{http://www.mat.univie.ac.at/\symbol{126}zweimueller/ }
\subjclass[2000]{Primary 28D05, 37A40, 37A50, 37E05, 60F05.}
\keywords{infinite invariant measure, limit distribution, null-recurrent processes, rare
events, indifferent fixed points, return-time statistics, hitting-time statistics}

\begin{abstract}
We study convergence of return- and hitting-time distributions of small sets
$E_{k}$ with $\mu(E_{k})\rightarrow0$ in recurrent ergodic dynamical systems
preserving an infinite measure $\mu$. Some properties which are easy in finite
measure situations break down in this null-recurrent setup. However, in the
presence of a uniform set $Y$ with wandering rate regularly varying of index
$1-\alpha$ with $\alpha\in(0,1]$, there is a scaling function suitable for all
subsets of $Y$. In this case, we show that return distributions for the
$E_{k}$ converge iff the corresponding hitting time distributions do, and we
derive an explicit relation between the two limit laws. Some consequences of
this result are discussed. In particular, this leads to improved sufficient
conditions for convergence to $\mathcal{E}^{1/\alpha}\,\mathcal{G}_{\alpha}$,
where $\mathcal{E}$ and $\mathcal{G}_{\alpha}$ are independent random
variables, with $\mathcal{E}$ exponentially distributed and $\mathcal{G}%
_{\alpha}$ following the one-sided stable law of order $\alpha$ (and
$\mathcal{G}_{1}:=1$). The same principle also reveals the limit laws
(different from the above) which occur at hyperblic periodic points of
prototypical null-recurrent interval maps. We also derive similar results for
the barely recurrent $\alpha=0$ case.

\end{abstract}
\maketitle

\section{Introduction}

The asymptotic behaviour of return- and hitting-time distributions of (very)
small sets in ergodic probability preserving dynamical systems has been
studied in great detail, and there is now a well-developed theory, both for
specific types of maps and sets, and for general abstract systems.

For infinite measure preserving situations, however, results are scarce. Only
recently some concrete classes of prototypical systems have been studied in
\cite{PS}, \cite{PS2}, and \cite{PSZ2}, where distributional limit theorems
for certain natural sequences of sets were established.

The purpose of the present article is to discuss some basic aspects of return-
and hitting-time limits for asymptotically rare events in the setup of
abstract infinite ergodic theory. After considering questions of scaling, we
discuss to what extent the natural relation between return- and hitting-time
limits which holds in the finite-measure setup carries over to infinite
measures, and how this can be used to prove convergence to specific laws.%
\newline
%

\noindent
\textbf{General setup. }Throughout, \emph{all measures are understood to be
}$\sigma$\emph{-finite}. We study \emph{measure preserving transformations}
$T$ (not necessarily invertible) on a measure space $(X,\mathcal{A},\mu)$.
Here $T$ will be \emph{ergodic} and also \emph{conservative} (meaning that
$\mu(A)=0$ for all wandering sets, that is, $A\in\mathcal{A}$ with $T^{-n}A$,
$n\geq0$, pairwise disjoint), and thus \emph{recurrent} (in that $A\subseteq%
{\textstyle\bigcup_{n\geq1}}
T^{-n}A$ mod $\mu$ for $A\in\mathcal{A}$). Our emphasis will be on the
\emph{infinite measure case}, $\mu(X)=\infty$.

For $T$ such a conservative ergodic measure preserving transformation
(\emph{c.e.m.p.t.}) on $(X,\mathcal{A},\mu)$, and any $Y\in\mathcal{A}$,
$\mu(Y)>0$, we define the \emph{first entrance time }function of $Y$,
$\varphi_{Y}:X\rightarrow\mathbb{N}\cup\{\infty\}$ by $\varphi_{Y}%
(x):=\inf\{n\geq1:T^{n}x\in Y\}$, $x\in X$, and let $T_{Y}x:=T^{\varphi(x)}x$,
$x\in X$. When restricted to $Y$, $\varphi_{Y}$ is called the \emph{first
return time }of $Y$, and $\mu\mid_{Y\cap\mathcal{A}}$ is invariant under the
ergodic \emph{first return map}, $T_{Y}\,$restricted to $Y$. If $\mu
(Y)<\infty$, it is natural to regard $\varphi_{Y}$ as a random variable on the
probability space $(X,\mathcal{A},\mu_{Y})$, where $\mu_{Y}(E):=\mu(Y)^{-1}%
\mu(Y\cap E)$. By Kac' formula, $\int\varphi_{Y}\,d\mu_{Y}=\mu(X)/\mu(Y)$. It
is well known (see \cite{A0}) that, for suitable reference sets $Y$, the
distribution of this variable reflects important features of the system
$(X,\mathcal{A},\mu,T)$.%
\newline
%

\noindent
\textbf{Return- and hitting-time distributions for small sets. }Rather than
focusing on a particular set $Y$, the present article studies the behaviour of
such distributions for sequences $(E_{k})$ of sets of (strictly)
\emph{positive} finite measure with $\mu(E_{k})\rightarrow0$. In particular,
these are sequences of \emph{asymptotically rare events}\footnote{One natural
definition of a sequence $(E_{k})_{k\geq1}$ of \emph{asymptotically rare
events} in a possibly infinite measure space $(X,\mathcal{A},\mu)$ is that
$\nu(E_{k})\rightarrow0$ for all probability measures $\nu\ll\mu$. See Remark
2.1 of \cite{Z2018} and \cite{ZFundMath2}.}. As return times to small sets
will typically be very large, the functions $\varphi_{E_{k}}$ need to be
normalized, which will be done using a certain scaling function $\gamma$.

We will thus study the distributions of random variables of the form
$\gamma(\mu(E))\,\varphi_{E}$ on $(E,E\cap\mathcal{A},\mu_{E})$, with $\mu(E)$
small, and call this the \emph{(normalized) return time distribution of} $E$,
\[
\mathsf{law}_{\mu_{E}}[\gamma(\mu(E))\,\varphi_{E}]
\]
where, for $\psi:X\rightarrow X^{\prime}$ any $\mathcal{A}$-$\mathcal{A}%
^{\prime}$-measurable map and $\nu\ll\mu$ a probability on $(X,\mathcal{A})$,
we write \textsf{law}$_{\nu}[\psi]:=\nu\circ\psi^{-1}$. In fact, we can use
any such $\nu$ as an initial distribution, in which case we refer to
\[
\mathsf{law}_{\nu}[\gamma(\mu(E))\,\varphi_{E}]
\]
as the \emph{(normalized) hitting time distribution of} $E$ \emph{(under }%
$\nu$\emph{)}. This leads to two different ways of looking at the
$\varphi_{E_{k}}$ for a sequence $(E_{k})$ as above: \emph{asymptotic return
distribution}s of $(E_{k})$ are limits, as $k\rightarrow\infty$, of
$(\mathsf{law}_{\mu_{E_{k}}}[\gamma(\mu(E_{k}))\,\varphi_{E_{k}}])_{k\geq1}$,
while \emph{asymptotic hitting distribution}s are limits of $(\mathsf{law}%
_{\nu}[\gamma(\mu(E_{k}))\,\varphi_{E_{k}}])_{k\geq1}$ for some fixed $\nu$.
(The latter limits do not depend on the choice of $\nu$, and we often take
$\nu=\mu_{Y}$ for some nice set $Y$, see Section \ref{Sec_RetVersHit} below.)
\emph{Understanding the relation between these two types of limits will be one
central theme of this article}.

It will be convenient to regard the distributions above as measures on
$[0,\infty]$. Accordingly, we let $\mathcal{F}:=\{F:[0,\infty)\rightarrow
\lbrack0,1]$, non-decreasing and right-continuous$\}$ be the set of
sub-probability distribution functions on $[0,\infty)$. For $F$, $F_{n}%
\in\mathcal{F}$ ($n\geq1$) we write $F_{n}\Rightarrow F$ for \emph{vague
convergence}, i.e. $F_{n}(t)\rightarrow F(t)$ at all continuity points of $F$.
For efficiency, we shall also use $F_{n}(t)\Longrightarrow F(t)$ to express
the same thing. (This allows us to use explicit functions of $t$.) If $\sup
F(t)=1$ this is the usual \emph{weak convergence} of probability distribution
functions on $[0,\infty)$.%
\newline
%

\noindent
\textbf{Pointwise dual ergodicity and uniform sets.} Some classes of
well-behaved infinite measure preserving systems are characterized by the
existence of distinguished \emph{reference sets} $Y$, $0<\mu(Y)<\infty$, with
special properties. Those are often defined in terms of the \emph{transfer
operator} $\widehat{T}:L_{1}(\mu)\rightarrow L_{1}(\mu)$, with $\int_{X}%
u\cdot(v\circ T)\,d\mu=\int_{X}\widehat{T}u\cdot v\,d\mu$ for all $u\in
L_{1}(\mu)$ and $v\in L_{\infty}(\mu)$. The operator $\widehat{T}$ naturally
extends to $\{u:X\rightarrow\lbrack0,\infty)$ $\mathcal{A}$-measurable$\}$. It
is a linear Markov operator, $\int_{X}\widehat{T}u\,d\mu=\int_{X}u\,d\mu$ for
$u\geq0$. The m.p.t. $T$ is conservative and ergodic if and only if
$\sum_{k\geq0}\widehat{T}^{k}u=\infty$ a.e. for all $u\in L_{1}^{+}%
(\mu):=\{u\in L_{1}(\mu):u\geq0$ and $\mu(u)>0\}$ or (equivalently) all
$u\in\mathcal{D}(\mu):=\{u\in L_{1}(\mu):u\geq0$, $\mu(u)=1\}$. By invariance
of $\mu$ we have $\widehat{T}1_{X}=1_{X}$.

A c.e.m.p.t. $T$ on the space $(X,\mathcal{A},\mu)$ is said to be
\emph{pointwise dual ergodic} (cf. \cite{A0}, \cite{A1}) if there is some
sequence $(a_{n})$ in $(0,\infty)$ such that
\begin{equation}
\frac{1}{a_{n}}\sum_{k=0}^{n-1}\widehat{T}^{k}u\longrightarrow\mu(u)\cdot
1_{X}\text{ \quad}%
\begin{array}
[c]{c}%
\text{a.e. on }X\text{ as }n\rightarrow\infty\text{, for every}\\
u\in L_{1}(\mu)\text{ with }\mu(u)\neq0\text{.}%
\end{array}
\label{Eq_Def_pde}%
\end{equation}
In this case, $(a_{n})$ (unique up to asymptotic equivalence, with
$a_{n}\rightarrow\infty$) is called a \emph{return sequence} of $T$. W.l.o.g.
we will assume throughout that $a_{n}=a_{T}(n)$ for some strictly increasing
continuous $a_{T}:[0,\infty)\rightarrow\lbrack0,\infty)$ with $a_{T}(0)=0$.
For convenience, we shall call any homeomorphism of $[0,\infty)$ a
\emph{scaling function}. Note that in case $\mu(X)=\infty$, we always have
$a_{T}(s)=o(s)$ as $s\rightarrow\infty$. Letting $b_{T}$ denote the inverse
function of $a_{T}$, we thus see that $s=o(b_{T}(s))$ as $s\rightarrow\infty$.
For later use we define, for a pointwise dual ergodic system, another scaling
function $\gamma_{T}:[0,\infty)\rightarrow\lbrack0,\infty)$ with $\gamma
_{T}(s)=o(s)$ as $s\searrow0$ via
\begin{equation}
\gamma_{T}(0):=0\text{ \quad and \quad}\gamma_{T}(s):=1/b_{T}(1/s)\text{ \quad
for }s>0\text{.}\label{Eq_DefGammaT}%
\end{equation}
By Egorov's theorem, the convergence in (\ref{Eq_Def_pde}) is uniform on
suitable sets (depending on $u$) of arbitrarily large measure. It is useful to
identify specific pairs $(u,Y)$, with $u\in\mathcal{D}(\mu)$ and
$Y\in\mathcal{A}$, $0<\mu(Y)<\infty$, such that
\begin{equation}
\left\Vert 1_{Y}\cdot\left(  \frac{1}{a_{n}}\sum_{k=0}^{n-1}\widehat{T}%
^{k}u-1_{X}\right)  \right\Vert _{\infty}\longrightarrow0\text{ \quad as
}n\rightarrow\infty\text{,}\label{Eq_Def_UnifSet}%
\end{equation}
in which case we shall refer to $Y$ as a $u$\emph{-uniform set} (compare
\cite{A0}, \cite{T4}). A set $Y$ which is $\mu(Y)^{-1}\cdot1_{Y}$-uniform is
called a \emph{Darling-Kac (DK) set}, cf. \cite{A0}, \cite{A2}. The existence
of a uniform set implies pointwise dual ergodicity (as in Proposition 3.7.5 of
\cite{A0}), and the $a_{n}$ in (\ref{Eq_Def_UnifSet}) then form a return sequence.

Several basic classes of infinite measure preserving systems, including Markov
shifts and other Markov maps with good distortion properties (see \cite{A0},
\cite{A2}, and \cite{T3}), as well as various non-Markovian interval maps (see
\cite{Z2}, \cite{Z3}), are known to possess DK-sets.

Finer probabilistic statements about pointwise dual ergodic systems usually
require $a_{T}$ to be \emph{regularly varying} with \emph{index} $\alpha
\in\lbrack0,1]$ (written $a_{T}\in\mathcal{R}_{\alpha}$), meaning that for
every $c>0$, $a_{T}(ct)/a_{T}(t)\rightarrow c^{\alpha}$ as $t\rightarrow
\infty$ (see \cite{BGT}). The asymptotics of $a_{T}$ is intimately related to
the return distribution \textsf{law}$_{\mu_{Y}}[\varphi_{Y}]$ of any of its
uniform sets $Y$: Write $q_{n}(Y):=\mu_{Y}(\varphi_{Y}>n)$, $n\geq0$, for the
\emph{tail probabilities} of $\varphi_{Y}$, and define the \emph{wandering
rate }$(w_{N}(Y))_{N\geq1}$\emph{\ of} $Y$ as the sequence of partial sums
$w_{N}(Y):=\mu(Y)\sum_{n=0}^{N-1}q_{n}(Y)$. By Theorem 5.1 of \cite{AZ}, any
uniform set $Y$ has \emph{minmal wandering rate} (meaning that \underline
{$\lim$}$_{N\rightarrow\infty}w_{N}(A)/w_{N}(Y)\geq1$ whenever $0<\mu
(A)<\infty$), and satisfies Aaronson's \emph{asymptotic renewal equation} (as
in Proposition 3.8.6 of \cite{A0}). Combining the latter with Karamata's
Tauberian Theorem (KTT, see Corollary 1.7.3 in \cite{BGT}) shows that for
$\alpha\in\lbrack0,1]$, one has
\begin{equation}
(w_{N})\in\mathcal{R}_{1-\alpha}\quad\text{iff\quad}a_{T}\in\mathcal{R}%
_{\alpha},
\end{equation}
in which case
\begin{equation}
a_{T}(n)\sim\frac{1}{\Gamma(2-\alpha)\Gamma(1+\alpha)}\,\frac{n}{w_{n}%
(Y)}\text{ \quad as }n\rightarrow\infty\text{.}\label{Eq_AsyRenewalEqn}%
\end{equation}
Whenever $a_{T}\in\mathcal{R}_{\alpha}$, we can and will assume that
$a_{T}(n)$ is given by the right-hand expression of (\ref{Eq_AsyRenewalEqn})
for some fixed uniform set $Y$, so that $(a_{T}(n)/n)_{n\geq1}$ is strictly
decreasing. The latter implies that $(a_{T}(n))_{n\geq1}$ is subadditive,
\begin{equation}
a_{T}(m+n)\leq a_{T}(m)+a_{T}(n)\text{ \quad for }m,n\geq1\text{.}%
\label{Eq_AnSubadd}%
\end{equation}
%

\vspace{0.3cm}%
%

\noindent
\textbf{The concrete limit theorems of \cite{PS}, \cite{PS2}, and
\cite{PSZ2}.} The results of \cite{PS} and \cite{PS2}, were the starting point
for the present investigation of return- and hitting-time limits in
null-recurrent situations. They apply to certain skew-products which are
\textquotedblleft barely recurrent\textquotedblright\ in that $(w_{N}%
)\in\mathcal{R}_{1}$ (corresponding to $\alpha=0$ above). In that case, only
the seriously distorted function $a_{T}(\varphi_{E})$ of the return-time
$\varphi_{E}$ can have a nontrivial limit, see the discussion at the end of
Section \ref{Sec_RetVersHit} below. For natural sequences $(E_{k})$ of sets,
those variables were shown to converge to the law with distribution function
$H^{\ast}(t):=t/(1+t)$, $t\geq0$.

The skew-product structure was exploited through the use of local limit
theorems. That approach has been extended to some (classical probabilistic)
$\alpha\in\lbrack0,1/2]$ situations in \cite{PSZ}, and further work on
skew-products has been done in \cite{Yassine1}. To go beyond skew-products and
the local limit technique, the notion of $\mathcal{U}$-uniform sets was
introduced (see Remark \ref{Rem_WhyBetter} below) in \cite{PSZ2}, which dealt
with $\alpha\in(0,1]$ situations. For certain natural sequences $(E_{k})$,
suitably normalized return- (and hitting-) times $\gamma(\mu(E_{k}%
))\,\varphi_{E_{k}}$ were shown to converge to a law best expressed as the
distribution of the random variable
\begin{equation}
\mathcal{H}_{\alpha}:=\mathcal{E}^{\frac{1}{\alpha}}\,\mathcal{G}_{\alpha
}\text{ ,\quad}\alpha\in(0,1]\text{,}%
\end{equation}
where $\mathcal{E}$ and $\mathcal{G}_{\alpha}$ are independent random
variables, with $\mathcal{E}$ exponentially distributed ($\Pr[\mathcal{E}%
>t]=e^{-t}$ for $t\geq0$) and $\mathcal{G}_{\alpha}$, $\alpha\in(0,1)$,
following the one-sided stable law of order $\alpha$ (so that $\mathbb{E}%
[\exp(-s\mathcal{G}_{\alpha})]=\exp(-s^{\alpha})$ for $s\geq0$), while
$\mathcal{G}_{1}=1$. We will use $H_{\alpha}(t):=\Pr[\mathcal{H}_{\alpha}\leq
t]$, $t\geq0$, to denote the distribution function of $\mathcal{H}_{\alpha}$.%
\newline
%

\noindent
\textbf{Outline of results.} In contrast to references \cite{PS}, \cite{PS2},
\cite{PSZ2} mentioned before, which study specific classes of systems and
particular types of sequences $(E_{k})$, the present paper discusses the
asymptotics of general asymptotically rare sequences $(E_{k})$ in an abstract setup.

We first discuss the basic question of how to normalize the functions
$\varphi_{E}$, and show that it is impossible to find a scaling function
$\gamma$ such that $\gamma(\mu(E))$ captures the order of magnitude of
$\varphi_{E}$ for all (small) sets $E$ (Theorem \ref{T_NoUniversalScale1}).
However, if $T$ admits a uniform set $Y$ with regularly varying return
sequence of index $\alpha>0$, then there is some $\gamma=\gamma_{T}$ which
works for every $E$ contained in $Y$ (Theorem \ref{T_TightnessInsideDK}). A
similar statement applies in the case $\alpha=0$ if we consider $a_{T}%
(\varphi_{E})$ (Theorem \ref{P_NonlinTightScale}).

Next, we turn to the relation between asymptotic return distributions and
asymptotic hitting distributions. In the finite measure case, it is well known
that one exists iff the other does, and that there is a simple explicit
relation between them (\cite{HaydnLacVai05}, see Section 4 below). In contrast
to this, we observe that there is no infinite measure preserving system in
which convergence of return-time distributions does, in general, entail
nontrivial asymptotics of hitting-time distributions (Theorem
\ref{Thm_SameRetDistButDivergentHitDist}).

Nonetheless, we prove (Theorem \ref{T_Ret_versus_Hit}) that, in the setup of
systems with a uniform set $Y$ and regular variation with $\alpha>0$, for
every asymptotically rare sequence $(E_{k})$ \emph{inside} $Y$, the
return-time distributions converge iff the hitting-time distributions
converge, and also clarify the relation between the respective limit laws. In
fact, we establish a stronger version of this principle which also covers the
distorted variables $a_{T}(\varphi_{E})$ of the $\alpha=0$ case, Theorem
\ref{T_Ret_versus_Hit2}, which is the central result of our paper.

The latter theorem allows us to precisely characterize (in Theorems
\ref{T_Cge_to_H_alpha} and \ref{T_Cge_to_H_alphathetaStar}) sequences
$(E_{k})$ inside $Y$, which exhibit convergence to the specific limit laws
$H_{\alpha} $ and $H^{\ast}$ that occurred in \cite{PS}, \cite{PS2}, and
\cite{PSZ2}. This also leads to an improved version of the abstract limit
theorem of \cite{PSZ2}, see Theorem \ref{T_SuffForCgeToHAlpha} and Remark
\ref{Rem_WhyBetter} below, which provides sufficient conditions for
convergence to the laws $H_{\alpha}$. We then go beyond these previously
encountered limit distributions, identifying new laws $H_{\alpha,\theta}$ for
asymptotic hitting times. Theorem \ref{T_Cge_to_H_alphatheta} gives abstract
sufficient conditions for their occurrence, and Example \ref{Ex_B} shows that
this result applies to cylinders shrinking to hyperbolic fixed points of
prototypical null-recurrent interval maps with indifferent fixed points.

We also extend the method to include the barely recurrent $\alpha=0$ case.
While previous results yield convergence to $H^{\ast}$ for specific classes of
systems with skew-product structure, we give sufficient conditions for
convergence to this law in the setup of general pointwise dual ergodic systems
with slowly varying return sequences in Theorem
\ref{T_SuffForCgeToHAlphaThetaStar}. The same result contains conditions for
convergence to new limit laws $H_{\theta}^{\ast}$, and Examples \ref{Ex_C} and
\ref{Ex_D} illustrate the occurrence of both $H^{\ast}$ and $H_{\theta}^{\ast
}$ in basic one-dimensional systems.%
\newline
%

\noindent
\textbf{Acknowledgement.} R.Z. thanks Jon Aaronson, Fran\c{c}oise P\`{e}ne,
Beno\^{\i}t Saussol, and Maximilian Thaler for highly motivating and inspiring
discussions.%
\newline

\section{How to normalize return-times of small sets\label{Sec_Scales}}

We collect some facts regarding the order of magnitude of a return-time
variable $\varphi_{E}$, focusing on its relation to the measure of the set
$E$. We first record some basic observations to point out some of the
difficulties which are inevitable when dealing with infinite measures. We then
formulate the main results of this section.%
\newline
%

\noindent
\textbf{Scaling return-times in finite measure systems.} As a warm-up, assume
first that $(X,\mathcal{A},\mu,T)$ is ergodic and measure preserving, with
$\mu(X)<\infty$. Kac' formula $\int_{E}\varphi_{E}\,d\mu_{E}=\mu(X)/\mu(E)$
for the expectation of the return-time of an arbitrary set $E\in\mathcal{A}$
with $\mu(E)>0$ not only shows that $\mu(E)\,\varphi_{E}$ is the canonical
choice if we wish to use \emph{normalized return times}, but also yields the
simple estimate $\mu_{E}(\mu(E)\,\varphi_{E}>t)\leq1/t$, $t>0$. The latter can
be read as an explicit version of the trivial statement that the family of all
normalized return distributions,
\begin{equation}
\left\{  \text{\textsf{law}}_{\mu_{E}}[\mu(E)\,\varphi_{E}]:E\in
\mathcal{A}\text{, }\mu(E)>0\right\}  \text{, is tight.}%
\label{Eq_TrivialTightnessFini}%
\end{equation}
We record an obvious consequence of this by also stating that for every
$\eta>0$,
\begin{equation}
\left\{  \text{\textsf{law}}_{\mu_{E}}[\varphi_{E}]:E\in\mathcal{A}\text{,
}\mu(E)\geq\eta\right\}  \text{ is tight.}\label{Eq_TghtLargeSetsFini}%
\end{equation}
The relevance of these trivialities for the present paper lies in the fact
that they both break down when $\mu(X)=\infty$.%
\newline
%

\noindent
\textbf{Scaling return-times in infinite measure systems - difficulties.} Now
let $(X,\mathcal{A},\mu,T)$ be a c.e.m.p.t. system with $\mu(X)=\infty$. We
are interested in the return distributions of sets of positive finite measure.
Kac' formula remains valid in that $\int_{E}\varphi_{E}\,d\mu_{E}=\infty$ for
every set $E\in\mathcal{A}$ with $0<\mu(E)<\infty$, but it no longer provides
us with a canonical normalization for $\varphi_{E}$. Indeed, the situation is
much more complicated than it is in the finite measure regime.

\begin{proposition}
[\textbf{Basic (non-)tightness properties of return distributions}%
]\label{P_NoTightnesssLargeSets}Let $T$ be a c.e.m.p.t. on $(X,\mathcal{A}%
,\mu)$ with $\mu(X)=\infty$. \newline\newline\textbf{a)} The family of return
distributions of large sets $E$ is not tight: For every $\eta>0$,
\begin{equation}
\left\{  \text{\textsf{law}}_{\mu_{E}}[\varphi_{E}]:E\in\mathcal{A}\text{,
}\mu(E)\geq\eta\right\}  \text{ is \emph{not} tight.}%
\end{equation}
\newline\textbf{b)} Locally, the family of return distributions of large sets
$E$ is tight: Let $Y\in\mathcal{A}$ with $0<\mu(Y)<\infty$. Then for every
$\eta>0$,
\begin{equation}
\left\{  \text{\textsf{law}}_{\mu_{E}}[\varphi_{E}]:E\in Y\cap\mathcal{A}%
\text{, }\mu(E)\geq\eta\right\}  \text{ is tight.}%
\end{equation}
\newline\textbf{c)} Even locally, the family of return distributions of
arbitrary sets $E$ with normalization $\mu(E)$ is not tight: Let
$Y\in\mathcal{A}$ with $0<\mu(Y)<\infty$. Then
\begin{equation}
\left\{  \text{\textsf{law}}_{\mu_{E}}[\mu(E)\,\varphi_{E}]:E\in
Y\cap\mathcal{A}\text{, }\mu(E)>0\right\}  \text{ is \emph{not} tight.}%
\end{equation}

\end{proposition}%

\vspace{0.1cm}%

Statement c) of the proposition confirms that $\mu(E)$ is not an appropriate
normalizing factor. In Theorem \ref{T_TightnessInsideDK} below we identify,
under additional assumptions, a scaling function $\gamma:[0,\infty
)\rightarrow\lbrack0,\infty)$ for which $\gamma(\mu(E))$ gives a suitable
normalization, at least inside certain reference sets $Y$. Call $\gamma$ a
\emph{tight scale for return times in} $Y$ if
\begin{equation}
\left\{  \text{\textsf{law}}_{\mu_{E}}[\gamma(\mu(E))\,\varphi_{E}]:E\in
Y\cap\mathcal{A}\text{, }\mu(E)>0\right\}  \text{ is tight.}%
\label{Eq_DefTightScale}%
\end{equation}
We will show first that any $Y$ admits a tight scale.

\begin{proposition}
[\textbf{Existence of tight scales for arbitrary }$Y$]%
\label{P_ExTightScalesAnyY}Let $T$ be a c.e.m.p.t. on $(X,\mathcal{A},\mu)$
and $Y\in\mathcal{A}$ with $0<\mu(Y)<\infty$. Then there exists a tight scale
$\gamma$ for return times in $Y$.
\end{proposition}

There is a good reason for restricting to subsets of some fixed $Y$ in the
definition of a tight scale: there never is a scaling function $\gamma$ which
works for every set $Y$ of positive finite measure. We shall prove

\begin{theorem}
[\textbf{No universal tight scale for return times}]%
\label{T_NoUniversalScale1}Let $T$ be a c.e.m.p.t. on $(X,\mathcal{A},\mu)$
with $\mu(X)=\infty$, and $\gamma$ a scaling function. Then there is some
$Y\in\mathcal{A}$ with $\mu(Y)=1$ such that $\gamma$ is \emph{not} a tight
scale for return times in $Y$.
\end{theorem}%

\vspace{0.3cm}%

Observe next that if $\gamma$ is such a tight scale for $Y$, then any scaling
function $\widetilde{\gamma}$ with $\widetilde{\gamma}(s)=o(\gamma(s))$ as
$s\searrow0$ kills return time functions of small sets in $Y$ in that
$\mu_{E_{k}}(\gamma(\mu(E_{k}))\,\varphi_{E_{k}}\leq t)\Rightarrow1 $ for all
$t>0$ whenever $(E_{k})$ is a sequence in $Y\cap\mathcal{A}$ with $\mu
(E_{k})\rightarrow0$. But, naturally, we will mostly be interested in a
scaling function $\gamma$ which is not only tight, but also a \emph{nontrivial
scale\ for return times in} $Y$ in that%
\begin{gather}
\text{there is a sequence }(E_{k})\text{ of asymptotically rare events in
}Y\text{ s.t.}\label{Eq_NonTrivialSequenceForGamma}\\
\underline{\lim}_{k\rightarrow\infty}\,\mu_{E_{k}}(\gamma(\mu(E_{k}%
))\,\varphi_{E_{k}}>t^{\ast})>0\text{ \quad for some }t^{\ast}>0\text{.}%
\nonumber
\end{gather}
For the sake of completeness we include the very easy observation that there
are always sets with exceptionally short returns, which elude any given scale function:

\begin{proposition}
[\textbf{Sets with very short returns}]\label{Prop_VeryTrivialProp}Let $T$ be
a c.e.m.p.t. on $(X,\mathcal{A},\mu)$, and $\gamma:[0,\infty)\rightarrow
\lbrack0,\infty)$ a scaling function. Assume that $Y\in\mathcal{A}$ satisfies
$\mu(Y)>0$, then there are sets $E_{k}\in Y\cap\mathcal{A}$, $k\geq1$, such
that $0<\mu(E_{k})\rightarrow0$ and
\begin{equation}
\mu_{E_{k}}(\gamma(\mu(E_{k}))\,\varphi_{E_{k}}\leq t)\longrightarrow1\text{
\quad for }t>0\text{.}\label{Eq_jvkdjhfasgfjadfsd}%
\end{equation}

\end{proposition}%

\vspace{0.3cm}%
%

\noindent
\textbf{Identifying nontrivial tight scales in the presence of regular
variation.} There are systems which possess distinguished reference sets $Y$
of positive finite measure for which we can explicitly identify a good scaling
function. The main positive result of the present section is

\begin{theorem}
[\textbf{Nontrivial tight scale in a uniform set, }$\alpha>0$]%
\label{T_TightnessInsideDK}Let $T$ be a c.e.m.p.t. on $(X,\mathcal{A},\mu)$,
$\mu(X)=\infty$. Assume that $Y\in\mathcal{A}$ is a uniform set and that
$a_{T}\in\mathcal{R}_{\alpha}$ for some $\alpha\in(0,1]$. Let $b_{T} $ be the
inverse function of $a_{T}$, and
\begin{equation}
\gamma_{T}(s):=1/b_{T}(1/s)\text{,\quad}s>0\text{.}%
\end{equation}
Then $\gamma_{T}\in\mathcal{R}_{1/\alpha}(0^{+})$ and $\gamma_{T}$ is a
nontrivial tight scale for return times in $Y$.
\end{theorem}%

\vspace{0.1cm}%

The situation inside a uniform set with regularly varying return sequence
therefore is not as wild as it is for arbitrary sets. The function $\gamma
_{T}$ is the scale which has been used in the concrete limit theorems of
\cite{PSZ}, and \cite{PSZ2}.

Note that the theorem does not cover the $\alpha=0$ case. Indeed, for
pointwise dual ergodic systems with $a_{T}\in\mathcal{R}_{0}$ it is more
natural to consider the nonlinear function $a_{T}(\varphi_{E})$ of
$\varphi_{E}$ rather than just rescaling $\varphi_{E}$ by a constant factor.
As shown in \cite{DE}, \cite{PS}, and \cite{PS2}, this often has a natural
limit distribution. The next result shows that $\mu(E)$ is once again a
natural normalization in this situation.

\begin{theorem}
[\textbf{Nontrivial tight distorted times in a uniform set,} $\alpha=0$%
]\label{P_NonlinTightScale}Let $T$ be a c.e.m.p.t. on $(X,\mathcal{A},\mu)$,
$\mu(X)=\infty$. Assume that $Y\in\mathcal{A}$ is a uniform set and that
$a_{T}\in\mathcal{R}_{0}$. Then,%
\begin{equation}
\left\{  \text{\textsf{law}}_{\mu_{E}}[\mu(E)\,a_{T}(\varphi_{E})]:E\in
Y\cap\mathcal{A}\text{, }\mu(E)>0\right\}  \text{ is tight.}%
\label{Eq_NonlinTightScale}%
\end{equation}
Moreover, there exists a sequence $(E_{k}^{\ast})$ in $Y\cap\mathcal{A}$ with
$0<\mu(E_{k}^{\ast})\rightarrow0$ for which
\begin{equation}
\mu_{E_{k}^{\ast}}(\mu(E_{k}^{\ast})\,a_{T}(\varphi_{E_{k}^{\ast}})\leq
t)\Longrightarrow1_{[1,\infty)}(t)\left(  1-\tfrac{1}{t}\right)  \text{ \quad
as }k\rightarrow\infty\text{.}\label{Eq_TheExtremeAlpha0Limit}%
\end{equation}

\end{theorem}

The proofs of Theorems \ref{T_TightnessInsideDK} and \ref{P_NonlinTightScale}
will depend on the investigation of the relations between return- and
hitting-time limits presented in Section \ref{Sec_RetVersHit}.

\section{Proofs for some results of Section 2}

We begin with the

\begin{proof}
[\textbf{Proof of Proposition \ref{P_NoTightnesssLargeSets}.}]\textbf{a)} We
show that there are $E_{k}\in\mathcal{A}$ with $\mu(E_{k})=\eta$ such that
$\varphi_{E_{k}}\geq k$ on $E_{k}$ for $k\geq1$.

Note first that an infinite measure space allowing a c.e.m.p. map $T$ is
necessarily nonatomic. Take some $Y\in\mathcal{A}$ with $0<\mu(Y)<\infty$, and
set $Y_{0}:=Y$ and $Y_{n}:=Y^{c}\cap\{\varphi_{Y}=n\}$, $n\geq1$. For any
$k\geq1$, the set $A_{k}:=\bigcup_{j\geq1}Y_{jk}$ satisfies $\varphi_{A_{k}%
}\mid_{A_{k}}\geq k$ and therefore $\varphi_{E}\mid_{E}\geq k$ holds for all
$E\in A_{k}\cap\mathcal{A}$. Since $\mu(A_{k})=\infty$ and $\mu$ is nonatomic,
$A_{k}$ has a subset $E_{k}$ with $\mu(E_{k})=\eta$.%
\newline
%

\noindent
\textbf{b) } We first prove that for every $E\in Y\cap\mathcal{A}$ with
$\mu(E)>0$, and any $m,n\geq1$,
\begin{equation}
\mu_{E}\left(  \varphi_{E}>mn\right)  \leq\frac{\mu(Y)}{\mu(E)}\left(
\frac{1}{m}+m\,\mu_{Y}\left(  \varphi_{Y}>n\right)  \right)  \text{.}%
\label{Eq_TheESTIMATE}%
\end{equation}
Note first that decomposing an excursion from $E$ into consecutive excursions
from $Y$, we can represent $\varphi_{E}$ as
\begin{equation}
\varphi_{E}=\sum_{j=0}^{\varphi_{E}^{Y}-1}\varphi_{Y}\circ T_{Y}^{j}\text{
\quad on }Y\text{,}\label{Eq_sdhfgjsdhgfjshdfjsafjlsagfhfgjsdhfjdshf}%
\end{equation}
where $\varphi_{E}^{Y}(x):=\inf\{i\geq1:T_{Y}^{i}x\in E\}$ denotes the first
entrance time of $E$ under the induced map $T_{Y}$. This reveals that
\begin{equation}
E\cap\{\varphi_{E}>mn\}\subseteq\left(  E\cap\{\varphi_{E}^{Y}>m\}\right)
\,\cup\,\left(  Y\cap\bigcup_{j=0}^{m-1}T_{Y}^{-j}\{\varphi_{Y}>n\}\right)
\text{.}%
\end{equation}
Applying Kac' formula to $T_{Y}$ gives $\mu_{E}(\varphi_{E}^{Y}>m)\leq
\mu(Y)/(m\mu(E))$, and since $T_{Y}$ preserves $\mu_{Y}$, it is clear that
$\mu_{Y}(\bigcup_{j=0}^{m-1}T_{Y}^{-j}\{\varphi_{Y}>n\})\leq m\,\mu_{Y}\left(
\varphi_{Y}>n\right)  $. Combining these yields (\ref{Eq_TheESTIMATE}).

Now take any $\varepsilon>0$. First choose $m\geq1$ so large that $\frac
{\mu(Y)}{\eta m}<\frac{\varepsilon}{2}$, then pick $n\geq1$ so large that
$\frac{\mu(Y)}{\eta}m\,\mu_{Y}\left(  \varphi_{Y}>n\right)  <\frac
{\varepsilon}{2}$ as well. Now (\ref{Eq_TheESTIMATE}) shows that $\mu
_{E}\left(  \varphi_{E}>mn\right)  <\varepsilon$ whenever $E\in Y\cap
\mathcal{A}$ satisfies $\mu(E)\geq\eta$.%
\newline
%

\noindent
\textbf{c) } Assume w.l.o.g. that $\mu(Y)=1$. We prove that there are sets
$E_{k}\in Y\cap\mathcal{A}$ with $0<\mu(E_{k})<1/k$ and
\begin{equation}
\mu_{E_{k}}\left(  \mu(E_{k})\,\varphi_{E_{k}}>k\right)  >(k-1)/k\text{ \quad
for }k\geq1\text{.}\label{Eq_golgolgol}%
\end{equation}
Fix any $k\geq1$. Since, according to Kac' formula, $\int_{Y}\varphi_{Y}%
\,d\mu=\infty$, we have $m^{-1}\sum_{j=0}^{m-1}\varphi_{Y}\circ T_{Y}%
^{j}\rightarrow\infty$ a.e. on $Y$ by (an obvious extension of) the ergodic
theorem. An Egorov-type argument shows that there is some $Z\in Y\cap
\mathcal{A}$ with $\mu(Y\setminus Z)<1/(2k)$ and an integer $M>k$ such that
\begin{equation}
\sum_{j=0}^{m-1}\varphi_{Y}\circ T_{Y}^{j}>2mk\text{ \quad on }Z\text{ for
}m\geq M\text{.}%
\end{equation}
Recalling the representation (\ref{Eq_sdhfgjsdhgfjshdfjsafjlsagfhfgjsdhfjdshf}%
), we conclude that for every $E\in Y\cap\mathcal{A}$,
\begin{equation}
\mu(E)\,\varphi_{E}>2M\mu(E)k\text{ \quad on }Z\cap\{\varphi_{E}^{Y}\geq
M\}\text{.}\label{Eq_sosososo}%
\end{equation}
Now appeal to the Rokhlin lemma to obtain some $F\in Y\cap\mathcal{A}$ for
which the sets $F_{i}:=T_{Y}^{-i}F$, $i\in\{0,\ldots,M-1\}$ are pairwise
disjoint, and $\mu(Y\setminus%
{\textstyle\bigcup\nolimits_{i=0}^{M-1}}
F_{i})<1/(M+1)$. In particular, $1/(M+1)<\mu(F_{i})\leq1/M$ for each $i$.
Observing that $\mu(%
{\textstyle\bigcup\nolimits_{i=0}^{M-1}}
F_{i}\setminus Z)<\mu(%
{\textstyle\bigcup\nolimits_{i=0}^{M-1}}
F_{i})/k$, we see that there is some $i_{0}\in\{0,\ldots,M-1\}$ for which
\begin{equation}
\mu_{F_{i_{0}}}(F_{i_{0}}\setminus Z)<1/k\text{.}\label{Eq_bvmnbcbmv}%
\end{equation}
Let $E_{k}:=F_{i_{0}}$, then $E_{k}\in Y\cap\mathcal{A}$ with $1/(M+1)<\mu
(E_{k})<1/k$. On the other hand, by the Rokhlin tower structure, we have
$\varphi_{E_{k}}^{Y}\geq M$ on $E_{k}$, and can thus employ the estimate
(\ref{Eq_sosososo}) to see that
\begin{equation}
\mu(E_{k})\,\varphi_{E_{k}}>k\text{ \quad on }Z\cap E_{k}\text{.}%
\end{equation}
In view of (\ref{Eq_bvmnbcbmv}), this implies our claim (\ref{Eq_golgolgol}).
\end{proof}%

\vspace{0.1cm}%

Given Proposition \ref{P_NoTightnesssLargeSets}, it is easy to proceed to the

\begin{proof}
[\textbf{Proof of Proposition \ref{P_ExTightScalesAnyY}}]By part b) of
Proposition \ref{P_NoTightnesssLargeSets} there is some non-decreasing
sequence $(\vartheta_{m})_{m\geq1}$ in $(0,\infty)$ such that $\mu_{E}%
(\varphi_{E}>\vartheta_{m})<1/m$ for all $E\in Y\cap\mathcal{A}$ with
$\mu(E)>1/m$. Choose some scaling function $\gamma$ for which $\gamma
(1/m)=1/(m\vartheta_{m+1})$, $m\geq1$. Then, for $E\in Y\cap\mathcal{A}$ with
$\mu(E)\in(\frac{1}{m+1},\frac{1}{m}]$,
\begin{align*}
\mu_{E}(\gamma(\mu(E))\,\varphi_{E}>1/m)  & \leq\mu_{E}(\gamma(1/m)\,\varphi
_{E}>1/m)\\
& =\mu_{E}(\varphi_{E}>\vartheta_{m+1})<1/m\text{.}%
\end{align*}
This easily implies that $\gamma$ is tight for $Y$: Fix any $\varepsilon>0$
and choose $M:=\left\lfloor 1/\varepsilon\right\rfloor +1$. For any $E\in
Y\cap\mathcal{A}$ with $\mu(E)<1/M$ there is some $m\geq M$ with $\mu
(E)\in(\frac{1}{m+1},\frac{1}{m}]$, and hence $\mu_{E}(\gamma(\mu
(E))\,\varphi_{E}>1)<\varepsilon$ by the above. On the other hand, part b) of
Proposition \ref{P_NoTightnesssLargeSets} provides us with some $K$ such that
$\mu_{E}(\gamma(\mu(E))\,\varphi_{E}>K)<\varepsilon$ whenever $\mu(E)\geq1/M$.
\end{proof}%

\vspace{0.1cm}%

\begin{proof}
[\textbf{Proof of Proposition \ref{Prop_VeryTrivialProp}.}]Assume w.l.o.g.
that $\mu(Y)=1$. We construct $E_{k}\in Y\cap\mathcal{A}$, $k\geq1$, s.t.
$\mu(E_{k})=1/k$ and $\mu_{E_{k}}(\varphi_{E_{k}}>1)\leq1/k$. Then $(E_{k})$
satisfies (\ref{Eq_jvkdjhfasgfjadfsd}).

Fix any $k\geq1$. By the Rokhlin lemma, there is some $F\in Y\cap\mathcal{A}$,
$\mu(F)>0$, such that the sets $F_{i}:=T_{Y}^{-i}F$, $i\in\{0,\ldots,k-1\}$
are pairwise disjoint. Since $\mu$ is nonatomic, there is some $F^{\prime}\in
F\cap\mathcal{A}$ such that $\mu(F^{\prime})=1/k^{2}$. Set $E_{k}:=%
{\textstyle\bigcup\nolimits_{i=0}^{k-1}}
T_{Y}^{-i}F^{\prime}$ (disjoint), then $\mu(E_{k})=1/k$, and $\varphi_{E_{k}%
}=1$ on $E_{k}\setminus F^{\prime}$.
\end{proof}%

\vspace{0.1cm}%

We prepare the proof of Theorem \ref{T_NoUniversalScale1} by recording two
easy lemmas. Recall that $q_{n}(Y):=\mu_{Y}(\varphi_{Y}>n)$, $n\geq0$, are the
tail probabilities of $\varphi_{Y}$ under $\mu_{Y}$.

\begin{lemma}
[\textbf{Sufficient conditions for non-trivial and non-tight scales}%
]\label{L_EasySufficientForNon}Let $T$ be a c.e.m.p.t. on $(X,\mathcal{A}%
,\mu)$ with $\mu(X)=\infty$, $Y\in\mathcal{A}$ a set with $0<\mu(Y)<\infty$,
and let $\gamma:[0,\infty)\rightarrow\lbrack0,\infty)$ be a scaling function.
\begin{equation}
\text{If }\underset{n\rightarrow\infty}{\overline{\lim}}n\,\gamma(\mu
(Y)q_{n}(Y))>0\text{, then }\gamma\text{ is a non-trivial scale for
}Y\label{Eq_EasySuffNonTriv}%
\end{equation}
in the sense of (\ref{Eq_NonTrivialSequenceForGamma}), and
\begin{equation}
\text{if }\underset{n\rightarrow\infty}{\overline{\lim}}n\,\gamma(\mu
(Y)q_{n}(Y))=\infty\text{, then }\gamma\text{ is not a tight scale for
}Y\text{.}\label{Eq_EasySuffNonTight}%
\end{equation}

\end{lemma}

\begin{proof}
For $n\geq1$ set $E_{n}^{\ast}:=Y\cap\{\varphi_{Y}>n\}\subseteq Y$, then
$\mu(E_{n}^{\ast})=\mu(Y)q_{n}(Y)$, while $\varphi_{E_{n}^{\ast}}\geq
\varphi_{Y}>n$ on $E_{n}^{\ast}$. Hence,
\[
\gamma(\mu(E_{n}^{\ast}))\,\varphi_{E_{n}^{\ast}}>n\,\gamma(\mu(Y)q_{n}%
(Y))\text{ \quad on }E_{n}^{\ast}\text{,}%
\]
and the implications (\ref{Eq_EasySuffNonTriv}) and (\ref{Eq_EasySuffNonTight}%
) follow at once.
\end{proof}

\begin{remark}
[\textbf{Non-triviality of }$\gamma_{T}$]Condition of
(\ref{Eq_EasySuffNonTriv}) immediately shows that \emph{under the assumptions
of Theorem \ref{T_TightnessInsideDK}, if }$\alpha\in(0,1)$\emph{, then
}$\gamma_{T}$\emph{\ is non-trivial}.

Indeed, the asymptotic relation (\ref{Eq_AsyRenewalEqn}) together with the
Monotone Density Theorem (Theorem 1.7.2 of \cite{BGT}, this uses $\alpha<1$)
shows that in this case $q_{n}(Y)\sim\kappa_{\alpha}/a_{T}(n)$ as
$n\rightarrow\infty$ with $\kappa_{\alpha}:=(1-\alpha)/(\Gamma(2-\alpha
)\Gamma(1+\alpha))\in(0,\infty)$. Therefore, $n\,\gamma_{T}(\mu(Y)q_{n}%
(Y))\rightarrow(\mu(Y)\kappa_{\alpha})^{1/\alpha}\in(0,\infty) $ as
$n\rightarrow\infty$.
\end{remark}%

\vspace{0.1cm}%

\begin{lemma}
[\textbf{Prescribing the measure of a set with given wandering rate}%
]\label{L_PrescribeWRandMu}Let $T$ be a c.e.m.p.t. on $(X,\mathcal{A},\mu)$
with $\mu(X)=\infty$, $Z\in\mathcal{A}$ a set with $0<\mu(Z)<\infty$.
\newline\newline\textbf{a)} For every $m\geq1$, we have $w_{N}(Z^{c}%
\cap\{\varphi_{Z}=m\})\sim w_{N}(Z)$ as $N\rightarrow\infty$.\newline%
\newline\textbf{b)} For every $\eta\in(0,\infty)$ there is some $Y\in
\mathcal{A}$ with $\mu(Y)=\eta$ and $w_{N}(Y)\sim w_{N}(Z)$.
\end{lemma}

\begin{proof}
\textbf{a)} Fix $m\geq1$ and note that $Z^{\prime}:=Z^{c}\cap\{\varphi
_{Z}=m\}\subseteq%
{\textstyle\bigcup\nolimits_{k=0}^{m}}
T^{-n}Z$, so that
\[
w_{N}(Z^{\prime})\leq w_{N}(%
{\textstyle\bigcup\nolimits_{k=0}^{m}}
T^{-n}Z)\sim w_{N}(Z)\text{.}%
\]
On the other hand, for $N>m$ we have
\[%
{\textstyle\bigcup\nolimits_{k=0}^{N-1}}
T^{-n}Z\subseteq\left(
{\textstyle\bigcup\nolimits_{k=0}^{m-1}}
T^{-n}Z\right)  \,\cup\,\left(
{\textstyle\bigcup\nolimits_{k=0}^{N-m}}
T^{-k}(Z^{\prime})\right)  \text{,}%
\]
and hence $w_{N}(Z)\leq\mu(%
{\textstyle\bigcup\nolimits_{k=0}^{m-1}}
T^{-n}Z)+w_{N-m}(Z^{\prime})\sim w_{N}(Z^{\prime})$, as required.\newline%
\newline\textbf{b)} Since $\mu(Z^{c}\cap\{\varphi_{Z}=m\})=\mu(Z\cap
\{\varphi_{Z}>m\})\rightarrow0$ as $m\rightarrow\infty$, statement a) shows
that there are arbitrarily small sets which asymptotically have the same
wandering rate as $Z$. We can therefore assume w.l.o.g. that $\mu(Z)\leq\eta$.

Now, as $%
{\textstyle\bigcup\nolimits_{k=0}^{L-1}}
T^{-k}Z\nearrow X$ (mod $\mu$), there is some integer $L\geq1$ for which
\[
\mu(%
{\textstyle\bigcup\nolimits_{k=0}^{L-1}}
T^{-k}Z)\leq\eta<\mu(%
{\textstyle\bigcup\nolimits_{k=0}^{L}}
T^{-k}Z)\text{.}%
\]
Since $\mu$ is necessarily non-atomic, there is some measurable $W\subseteq
T^{-L}Z\setminus%
{\textstyle\bigcup\nolimits_{k=0}^{L-1}}
T^{-k}Z$ for which $\mu(W)=\eta-\mu(%
{\textstyle\bigcup\nolimits_{k=0}^{L-1}}
T^{-k}Z)$. Set $Y:=%
{\textstyle\bigcup\nolimits_{k=0}^{L-1}}
T^{-k}Z\cup W$, then $\mu(Y)=\eta$ and since $%
{\textstyle\bigcup\nolimits_{k=0}^{L-1}}
T^{-k}Z\subseteq Y\subseteq%
{\textstyle\bigcup\nolimits_{k=0}^{L}}
T^{-k}Z$ we have $w_{N}(Y)\sim w_{N}(Z)$.
\end{proof}%

\vspace{0.1cm}%

\begin{proof}
[\textbf{Proof of Theorem \ref{T_NoUniversalScale1}.}]In view of Lemma
\ref{L_EasySufficientForNon} it suffices to show that there is some
$Y\in\mathcal{A}$ with $\mu(Y)=1$ such that%
\begin{equation}
\overline{\lim}_{n\rightarrow\infty}\,n\,\gamma(q_{n}(Y))=\infty
\text{.}\label{Eq_ghjdhfgfdjgbcuagbfc}%
\end{equation}
Observe first that we can choose a sequence $(M_{n})_{n\geq1}$ in $(0,\infty)$
such that $M_{n}\nearrow\infty$ while $M_{n}/n\searrow0$ so slowly that
$a_{N}^{\circ}:=\sum_{n=1}^{N}\gamma^{-1}(M_{n}/n)\nearrow\infty$ as
$N\rightarrow\infty$. Since $\gamma^{-1}(M_{n}/n)\searrow0$, we also have
$a_{N}^{\circ}/N\searrow0$ in this case.

By Proposition 3.8.2 of \cite{A0}, there is some $Z\in\mathcal{A}$,
$0<\mu(Z)<\infty$, such that $w_{N}(Z)\geq2a_{N}$ for $N\geq1$. According to
Lemma \ref{L_PrescribeWRandMu}, there is some $Y\in\mathcal{A}$, $\mu(Y)=1$,
with $w_{N}(Y)\sim w_{N}(Z)$. In particular, there is some $N^{\circ}$ such
that $w_{N}(Y)>3a_{N}/2$ for $N\geq N^{\circ}$, that is,
\[%
{\textstyle\sum_{n=0}^{N}}
q_{n}(Y)>\frac{3}{2}%
{\textstyle\sum_{n=1}^{N}}
\gamma^{-1}(M_{n}/n)\text{ \quad for }N\geq N^{\circ}\text{.}%
\]
But since $a_{N}^{\circ}\nearrow\infty$ this implies that there are indices
$n_{k}\nearrow\infty$ for which
\[
q_{n_{k}}(Y)>\gamma^{-1}(M_{n_{k}}/n_{k})\text{ \quad whenever }k\geq1\text{,}%
\]
and hence $n_{k}\,\gamma(q_{n_{k}}(Y))>M_{n_{k}}$ for $k\geq1$. As $M_{n_{k}%
}\nearrow\infty$, this proves (\ref{Eq_ghjdhfgfdjgbcuagbfc}).
\end{proof}%

\vspace{0.1cm}%

\section{Return-time limits versus hitting-time limits\label{Sec_RetVersHit}}

This section contains the central results of the present paper. Our proofs of
several other results (including Theorem \ref{T_TightnessInsideDK} above) rely
on them.

We clarify the relation between asymptotic return-time distributions and
asymptotic hitting-time distributions in the situation of Theorems
\ref{T_TightnessInsideDK} and \ref{P_NonlinTightScale}, where suitable scaling
functions can be found.%
\newline
%

\noindent
\textbf{Strong distributional convergence of hitting times.} In contrast to
asymptotic return distributions, asymptotic hitting distributions are robust
under any change of (absolutely continuous) initial measure. This property was
mentioned before. Since is plays a crucial role in the following, we provide a
formal statement (taken from \cite{Z7}, Corollary 5).

\begin{lemma}
[\textbf{Strong distributional convergence of hitting times}]%
\label{L_StrongDCge}Let $T$ be a c.e.m.p.t. on $(X,\mathcal{A},\mu)$, $\gamma$
a scaling function, and $(E_{k})$ a sequence in $\mathcal{A}$ with
$0<\mu(E_{k})\rightarrow0$. If there are some $F\in\mathcal{F}$ and some
probability $\nu^{\ast}\ll\mu$ such that
\begin{equation}
\nu^{\ast}(\gamma(\mu(E_{k}))\,\varphi_{E_{k}}\leq t)\Longrightarrow
F(t)\text{ \quad as }k\rightarrow\infty\text{,}%
\end{equation}
then
\begin{equation}
\nu(\gamma(\mu(E_{k}))\,\varphi_{E_{k}}\leq t)\Longrightarrow F(t)\text{ \quad
as }k\rightarrow\infty\text{ \quad for all }\nu\ll\mu\text{.}%
\end{equation}

\end{lemma}

The reason behind this principle is \emph{asymptotic} $T$\emph{-invariance in
measure} of the variables $\gamma(\mu(E_{k}))\,\varphi_{E_{k}}$, which means
that
\begin{equation}
\gamma(\mu(E_{k}))\,\varphi_{E_{k}}-\gamma(\mu(E_{k}))\,\varphi_{E_{k}}\circ
T\overset{\nu}{\longrightarrow}0\text{ \quad as }k\rightarrow\infty
\end{equation}
for all probabilities $\nu\ll\mu$. The latter is immediate from the fact that
\begin{equation}
\varphi_{E}=\varphi_{E}\circ T+1\text{ \quad on }\{\varphi_{E}>1\}=T^{-1}%
E\text{.}\label{Eq_nbmcxnbhdfvbmchvvrhf}%
\end{equation}
Arguments related to those responsible for this lemma will be used in Section
\ref{Sec_LimitThmAlphaPositive}.%
\newline
%

\noindent
\textbf{Hitting versus returning in finite measure systems.} Recall that in
the probability-preserving setup, a simple general principle relates limit
laws for normalized hitting-times and for normalized return-times to each
other (see \cite{HaydnLacVai05}):%
\newline
%

\noindent
\textbf{Theorem HLV (Return and hitting-time limits for finite measure).
}\textit{Let }$T$\textit{\ be an ergodic m.p.t. on }$(X,A,\mu)$\textit{, }%
$\mu(X)=1$\textit{. Suppose that }$E_{k}$\textit{, }$k\geq1$\textit{, are sets
of positive measure with }$\mu(E_{k})\rightarrow0$\textit{.}

\textit{Then the normalized return-time distributions of the }$E_{k}%
$\textit{\ converge in that\ }%
\begin{equation}
\mu_{E_{k}}(\mu(E_{k})\,\varphi_{E_{k}}\leq t)\Longrightarrow\widetilde
{F}(t)\text{ \quad as }k\rightarrow\infty
\end{equation}
\textit{for some }$\widetilde{F}\in\mathcal{F}$\textit{, if and only if the
normalized hitting-time distributions converge, }%
\begin{equation}
\mu(\mu(E_{k})\,\varphi_{E_{k}}\leq t)\Longrightarrow F(t)\text{ \quad as
}k\rightarrow\infty
\end{equation}
\textit{for some }$F\in\mathcal{F}$\textit{. In this case the limit laws
satisfy }%
\begin{equation}
F(t)=\int_{0}^{t}[1-\widetilde{F}(s)]\,ds\text{ \quad for }t\geq
0\text{.}\label{Eq_HLV_relation1}%
\end{equation}

If (\ref{Eq_HLV_relation1}) is satisfied, $F$ is sometimes called the
\emph{integrated tail distribution} of $\widetilde{F}$.

\begin{remark}
\label{R_DegenerateDegenerate}In view of Kac' formula, it is a priori clear
that any limit law $\widetilde{F}$ of the normalized return times has
expectation not exceeding\thinspace$1$. On the other hand, it can (and often
does) happen that the limit law is concentrated at $t=0$, so that
$\widetilde{F}=1_{[0,\infty)}$. According to (\ref{Eq_HLV_relation1}), this
happens iff the scaled hitting-times $\mu(E_{k})\,\varphi_{E_{k}}$ diverge to
infinity under $\mu$, that is, iff $F=0$.
\end{remark}%

\vspace{0.2cm}%
%

\noindent
\textbf{Hitting versus returning in infinite measure systems - difficulties.}
We address the obvious question of how the two different types of limit
theorems are related in null-recurrent situations. As a first complication,
meaningful limits must involve some nontrivial scaling function $\gamma$, and
we have seen above that there is no natural choice which works for all small
sets. But even if we can find $\gamma$ such that a specific sequence $(E_{k})$
has a nondegenerate return-time limit on this scale, this does not imply that
it also has a hitting-time limit different from $\infty$ if we use $\gamma$
(in contrast to the finite measure case, compare Remark
\ref{R_DegenerateDegenerate}.)

\begin{theorem}
[\textbf{All return time limits are compatible with exploding hitting times}%
]\label{Thm_SameRetDistButDivergentHitDist}Let $T$ be a c.e.m.p.t. on
$(X,\mathcal{A},\mu)$ with $\mu(X)=\infty$, and take any $Y\in\mathcal{A}$
with $0<\mu(Y)<\infty$. Let $\gamma$ be any scaling function.

Suppose that $E_{k}\in\mathcal{A}$, $k\geq1$, are sets of positive measure
with $\mu(E_{k})\rightarrow0$. If, for some $\widetilde{F}\in\mathcal{F}$,
\begin{equation}
\mu_{E_{k}}(\gamma(\mu(E_{k}))\,\varphi_{E_{k}}\leq t)\Longrightarrow
\widetilde{F}(t)\text{ \quad as }k\rightarrow\infty\text{,}%
\label{Eq_dbvazuegwauzeuzerrtrtrttrtrtrttrttrt}%
\end{equation}
then there exist $E_{k}^{\prime}\in\mathcal{A}$ with $\mu(E_{k}^{\prime}%
)=\mu(E_{k})$ and
\begin{equation}
\mu_{E_{k}^{\prime}}(\gamma(\mu(E_{k}^{\prime}))\,\varphi_{E_{k}^{\prime}}\leq
t)\Longrightarrow\widetilde{F}(t)\text{ \quad as }k\rightarrow\infty\text{,}%
\end{equation}
while
\begin{equation}
\mu_{Y}(\gamma(\mu(E_{k}^{\prime}))\,\varphi_{E_{k}^{\prime}}\leq
t)\Longrightarrow0\text{ \quad as }k\rightarrow\infty\text{.}%
\end{equation}

\end{theorem}

(Recall that by Helly's selection theorem, every asymptotivally rare sequence
$(E_{k})$ has a subsequence along which return distibutions converge as in
(\ref{Eq_dbvazuegwauzeuzerrtrtrttrtrtrttrttrt}).)

The (simple) principle behind this theorem is that we can move copies
$E_{k}^{\prime}$ of the $E_{k}$, which have the same return distributions, to
places which are far away from any given $Y$. Therefore we can only hope for a
result parallel to Theorem HLV if we focus on sequences $(E_{k})$ inside some
reference set.%

\vspace{0.2cm}%
%

\noindent
\textbf{Hitting versus returning inside a uniform set.} We prove an abstract
result in the spirit of Theorem HLV which applies to infinite measure
preserving maps possessing a uniform set with regularly varying return
sequence. We use the normalization discussed in the preceding section. The
following result confirms once again that the latter is a sensible choice.

\begin{theorem}
[\textbf{Return- versus hitting-time limits in uniform sets}]%
\label{T_Ret_versus_Hit}Let $T$ be a c.e.m.p.t. on $(X,\mathcal{A},\mu)$,
$\mu(X)=\infty$. Assume that $Y\in\mathcal{A}$ is a uniform set and that
$a_{T}\in\mathcal{R}_{\alpha}$ for some $\alpha\in(0,1]$. Define $\gamma_{T}$
by $\gamma_{T}(s):=1/b_{T}(1/s)$, $s>0$.

Suppose that $E_{k}\subseteq Y$, $k\geq1$, are sets of positive measure with
$\mu(E_{k})\rightarrow0$. Then the normalized return-time distributions of the
$E_{k}$ converge in that
\begin{equation}
\mu_{E_{k}}(\gamma_{T}(\mu(E_{k}))\,\varphi_{E_{k}}\leq t)\Longrightarrow
\widetilde{F}(t)\text{ \quad as }k\rightarrow\infty\label{Eq_AbstrDistrCgeII}%
\end{equation}
for some $\widetilde{F}\in\mathcal{F}$, if and only if the normalized
hitting-time distributions converge,
\begin{equation}
\mu_{Y}(\gamma_{T}(\mu(E_{k}))\,\varphi_{E_{k}}\leq t)\Longrightarrow
F(t)\text{ \quad as }k\rightarrow\infty\label{Eq_AbstrDistrCgeIII}%
\end{equation}
for some $F\in\mathcal{F}$. In this case the limit laws satisfy
\begin{equation}
F(t)=\int_{0}^{t}[1-\widetilde{F}(s)]\,\alpha\left(  t-s\right)  ^{\alpha
-1}ds\text{ \quad for }t\geq0\text{.}\label{Eq_MagicFormula1}%
\end{equation}

\end{theorem}%

\vspace{0.2cm}%

\begin{remark}
[\textbf{Some comments and consequences}]\label{Rem_Comments1}We record the
following:\newline\textbf{a)} In the $\alpha=1$ case of a \textquotedblleft
barely infinite\textquotedblright\ measure, (\ref{Eq_MagicFormula1}) reduces
to (\ref{Eq_HLV_relation1}).\newline\textbf{b)} In (\ref{Eq_MagicFormula1}),
the function $\widetilde{F}$ clearly determines $F$, and vice versa.\newline%
\textbf{c)} Relation (\ref{Eq_MagicFormula1}) shows that $F$ is necessarily
continuous on $[0,\infty)$ with $F(0)=0$. Moreover, since $\int_{0}^{t}%
\alpha\left(  t-s\right)  ^{\alpha-1}ds=t^{\alpha}\rightarrow\infty$ as
$t\rightarrow\infty$, it is immediate that $\widetilde{F}(s)\rightarrow1$ as
$s\rightarrow\infty$, so that $\widetilde{F}$ is a \emph{probability}
distribution function on $[0,\infty)$. \newline\textbf{d)} By the previous
remark and Theorem \ref{T_NoUniversalScale1} we see that under the assumptions
of Theorem \ref{T_Ret_versus_Hit} there are always sets $Y^{\prime}%
\in\mathcal{A}$, $0<\mu(Y^{\prime})<\infty$, inside which the conclusion of
Theorem \ref{T_Ret_versus_Hit} fails. \newline\textbf{e)} If $(E_{k})$ is a
sequence such that (\ref{Eq_AbstrDistrCgeII}) takes place with $\widetilde
{F}=1_{[0,\infty)}$, that is \textsf{law}$_{\mu_{E_{k}}}[\gamma_{T}(\mu
(E_{k}))\,\varphi_{E_{k}}]\Longrightarrow\delta_{0}$ (by Proposition
\ref{Prop_VeryTrivialProp} the set $Y$ always contains such a sequence), then
(\ref{Eq_AbstrDistrCgeIII}) holds with $F=0$, so that \textsf{law}$_{\mu_{Y}%
}[\gamma_{T}(\mu(E_{k}))\,\varphi_{E_{k}}]\Longrightarrow\delta_{\infty}%
$.\newline\textbf{f)} In (\ref{Eq_AbstrDistrCgeIII}) the measure $\mu_{Y}$ can
be replaced by any fixed probability measure $\nu\ll\mu$, see Lemma
\ref{L_StrongDCge}.
\end{remark}%

\vspace{0.2cm}%
%

\noindent
\textbf{Including the barely recurrent case: distorted return- and hitting
times.} Theorem \ref{T_Ret_versus_Hit} is a consequence of the following
result where return- and hitting times are distorted by the nonlinear function
$a_{T}$. Theorem \ref{T_Ret_versus_Hit2} is more general in that it also gives
nontrivial information about the \emph{barely recurrent}\ $\alpha=0$ case
which Theorem \ref{T_Ret_versus_Hit} does not cover. This case is of interest,
because of very natural examples (recurrent random walks on $\mathbb{Z}^{2}$
and recurrent $\mathbb{Z}^{2}$-extensions including the Lorentz process, see
\cite{PS} and \cite{PS2}, and slowly recurrent random walks on $\mathbb{R}$,
\cite{PSZ}) in which (for natural sequences $(E_{k})$) the $\mu(E_{k}%
)\,a_{T}(\varphi_{E_{k}})$ have been shown to converge to the limit law with
distribution function $H^{\ast}(t):=t/(1+t)$, $t\geq0$.

\begin{theorem}
[\textbf{Distorted return- and hitting-time limits in uniform sets}%
]\label{T_Ret_versus_Hit2}Let $T$ be a c.e.m.p.t. on $(X,\mathcal{A},\mu)$,
$\mu(X)=\infty$. Assume that $Y\in\mathcal{A}$ is a uniform set and that
$a_{T}\in\mathcal{R}_{\alpha}$ for some $\alpha\in\lbrack0,1]$.

Suppose that $E_{k}\subseteq Y$, $k\geq1$, are sets of positive measure with
$\mu(E_{k})\rightarrow0$. Then the distorted return-time distributions of the
$E_{k}$ converge,
\begin{equation}
\mu_{E_{k}}(\mu(E_{k})\,a_{T}(\varphi_{E_{k}})\leq t)\Longrightarrow
\widetilde{G}(t)\text{ \quad as }k\rightarrow\infty\label{Eq_AbstrDistrCgeI}%
\end{equation}
for some $\widetilde{G}\in\mathcal{F}$, if and only if the distorted
hitting-time distributions converge,
\begin{equation}
\mu_{Y}(\mu(E_{k})\,a_{T}(\varphi_{E_{k}})\leq t)\Longrightarrow G(t)\text{
\quad as }k\rightarrow\infty\label{Eq_AbstrDistrCgeIIIa}%
\end{equation}
for some $G\in\mathcal{F}$. In this case the limit laws satisfy, for $t\geq
0$,
\begin{equation}
G(t)=\left\{
\begin{array}
[c]{ll}%
t%
{\displaystyle\int_{0}^{1}}
[1-\widetilde{G}\left(  t\,(1-r)^{\alpha}\right)  ]\,\alpha r^{\alpha-1}\,dr &
\text{if }\alpha\in(0,1]\text{,}\\
& \\
t\,[1-\widetilde{G}(t)] & \text{if }\alpha=0\text{,}%
\end{array}
\right. \label{Eq_MagicFormula2}%
\end{equation}
and $G$ is Lipschitz continuous on $[0,\infty)$.
\end{theorem}%

\vspace{0.2cm}%

The proofs of the two theorems are given in the next section.

\begin{remark}
For $\alpha\in(0,1]$, the statement of Theorem \ref{T_Ret_versus_Hit2} is
equivalent to Theorem \ref{T_Ret_versus_Hit}, with $(\widetilde{G}%
(t),G(t))=(\widetilde{F}(t^{1/\alpha}),F(t^{1/\alpha}))$, see Lemma
\ref{L_RegVarDistortion} below. Remark \ref{Rem_Comments1} translates accordingly.
\end{remark}

Regarding the $\alpha=0$ case, some related facts are recorded in

\begin{remark}
[\textbf{More comments}]\label{Rem_Comments2}\textbf{a)} If $G\in\mathcal{F}$
satisfies
\begin{equation}
G(t)=t[1-G(t)]\text{ \quad for }t\geq0\text{,}%
\end{equation}
then $G=H^{\ast}$, which is the $\alpha=0$ limit law from \cite{PS},
\cite{PS2}, and \cite{PSZ}. \newline\textbf{b)} If $G,\widetilde{G}%
\in\mathcal{F}$ satisfy $G(t)=t\,[1-\widetilde{G}(t)]$ for $t\geq0$, as in the
$\alpha=0$ case of (\ref{Eq_MagicFormula2}), then $\widetilde{G}$ is
necessarily a probability distribution function on $[0,\infty)$, because
$G\leq1$. Moreover, in the $\alpha=0$ case, $\widetilde{G}$ is continuous on
$[0,\infty)$ since $G$ is.\newline\textbf{c)} Regarding the right-hand side of
(\ref{Eq_MagicFormula2}), note that for every $\widetilde{G}\in\mathcal{F}$
and every continuity point $t>0 $ of $\widetilde{G}$ we have
\[
\int_{0}^{1}[1-\widetilde{G}\left(  t\,(1-r)^{\alpha}\right)  ]\,\alpha
r^{\alpha-1}\,dr\longrightarrow1-\widetilde{G}(t)\text{ \quad as }%
\alpha\searrow0\text{.}%
\]
\newline
\end{remark}

Concerning changes of measure for distorted hitting times, we record the
following counterpart to Lemma \ref{L_StrongDCge}.

\begin{lemma}
[\textbf{Strong distributional convergence distorted hitting times}%
]\label{L_StrongDCgeBarely}Let $T$ be a pointwise dual ergodic c.e.m.p.t. on
$(X,\mathcal{A},\mu)$ and $(E_{k})$ a sequence in $\mathcal{A}$ with
$0<\mu(E_{k})\rightarrow0$. If there are some $G\in\mathcal{F}$ and some
probability $\nu^{\ast}\ll\mu$ s.t.
\begin{equation}
\nu^{\ast}(\mu(E_{k})\,a_{T}(\varphi_{E_{k}})\leq t)\Longrightarrow G(t)\text{
\quad as }k\rightarrow\infty\text{,}%
\end{equation}
then
\begin{equation}
\nu(\mu(E_{k})\,a_{T}(\varphi_{E_{k}})\leq t)\Longrightarrow G(t)\text{ \quad
as }k\rightarrow\infty\text{ \quad for all }\nu\ll\mu\text{.}%
\end{equation}

\end{lemma}

\begin{proof}
According to Theorem 1 of \cite{Z7} we need only check that the sequence
$(\mathsf{R}_{k})_{k\geq1}$ with $\mathsf{R}_{k}:=\mu(E_{k})\,a_{T}%
(\varphi_{E_{k}})$ is asymptotically $T$-invariant in measure in that
$\mathsf{R}_{k}-\mathsf{R}_{k}\circ T\longrightarrow^{\nu}0$ for every finite
$\nu\ll\mu$. But (\ref{Eq_nbmcxnbhdfvbmchvvrhf}) and (\ref{Eq_AnSubadd}), show
that
\begin{align*}
\mathsf{R}_{k}-\mathsf{R}_{k}\circ T  & =\mu(E_{k})\left(  a_{T}%
(\varphi_{E_{k}}\circ T+1)-a_{T}(\varphi_{E_{k}}\circ T)\right) \\
& \leq\mu(E_{k})\,a_{T}(1)\text{ \quad on }\{\varphi_{E_{k}}>1\}=T^{-1}%
E_{k}^{c}\text{,}%
\end{align*}
and the required property follows since $\mu(E_{k})\rightarrow0$.
\end{proof}%

\vspace{0.2cm}%

While the $a_{T}(\varphi_{E_{k}})$ do exhibit nontrivial asymptotic
distributional behaviour in the interesting $\alpha=0$ situations mentioned
above, the original $\varphi_{E_{k}}$ do not. (Hence, studying the
$a_{T}(\varphi_{E_{k}})$ is the right thing to do.) A formal version of this
statement is immediate from the following general fact.

\begin{proposition}
[\textbf{No way back from} $\ell(R_{n})$ \textbf{to} $R_{n}$]%
\label{P_NoWayBack}Let $\ell:[0,\infty)\rightarrow\lbrack0,\infty)$ be a
slowly varying homeomorphism. Assume that $(R_{n})$ is a sequence of
$[0,\infty]$-valued random variables with $R_{n}\Longrightarrow\infty$ for
which $(\ell(R_{n}))$ has a continuous limit distribution on $(0,\infty)$,
that is, there are $(\eta_{n})$ in $(0,\infty)$ and a continuous random
variable $L$ with $0<L<\infty$ a.s., such that
\begin{equation}
\eta_{n}\,\ell(R_{n})\Longrightarrow L\text{ \quad as }n\rightarrow
\infty\text{.}\label{Eq_StronglyDistortedCge}%
\end{equation}
Then any limit distribution of $(R_{n})$ is concentrated on $\{0,\infty\}$,
that is, if
\begin{equation}
\gamma_{n}\,R_{n}\Longrightarrow R\text{ \quad as }n\rightarrow\infty
\text{,}\label{Eq_OrdinaryCge}%
\end{equation}
for some $(\gamma_{n})$ in $(0,\infty)$ and some random variable $R$, then
$\Pr[R\in\{0,\infty\}]=1$.
\end{proposition}%

\vspace{0.2cm}%

\section{Proofs for Sections 2 and 4\label{Sec_Pf_Relation_Ret_vs_Hit}}

\textbf{General setup.} As a warm-up we provide the

\begin{proof}
[\textbf{Proof of Proposition \ref{P_NoWayBack}}]Assume, for a contradiction,
that (\ref{Eq_OrdinaryCge}) holds and $\Pr[e^{-c}<R\leq e^{c}]>0$ for some
$c>1$. Then there are $\kappa>0$ and $n_{0}\geq1$ such that the events
$A_{n}:=\{e^{-2c}<\gamma_{n}\,R_{n}\leq e^{2c}\}$ satisfy $\Pr[A_{n}%
]\geq\kappa$ for $n\geq n_{0}$. Set $\overline{\eta}_{n}:=1/\ell(1/\gamma
_{n})$ and $t_{n}^{\pm}:=\overline{\eta}_{n}\ell(e^{\pm2c}/\gamma_{n})$, so
that the above becomes
\begin{equation}
\Pr[t_{n}^{-}<\overline{\eta}_{n}\,\ell(R_{n})\leq t_{n}^{+}]\geq\kappa\text{
\quad for }n\geq n_{0}\text{.}\label{Eq_GanzEinfach}%
\end{equation}
Since $R_{n}\Longrightarrow\infty$, we have $\gamma_{n}\rightarrow0$, and slow
variation of $\ell$ yields $t_{n}^{\pm}\rightarrow1$.

Passing to a subsequence if necessary, we may assume that there is some
$[0,\infty]$-valued random variable $\overline{L}$ such that
\begin{equation}
\overline{\eta}_{n}\,\ell(R_{n})\Longrightarrow\overline{L}\text{ \quad as
}n\rightarrow\infty\text{.}%
\end{equation}
Now (\ref{Eq_GanzEinfach}) and $t_{n}^{\pm}\rightarrow1$ clearly imply
$\Pr[\overline{L}=1]\geq\kappa>0$. In view of the standard \emph{convergence
of types} theorem, applied to $(\ell(R_{n}))$, this contradicts our assumption
(\ref{Eq_StronglyDistortedCge}), where $L$ is a \emph{continuous} variable in
$(0,\infty)$.
\end{proof}%

\vspace{0.2cm}%

Throughout the rest of this section, assume (as in the theorems) that $T$ be a
c.e.m.p.t. on $(X,\mathcal{A},\mu)$, $\mu(X)=\infty$, and that $0<\mu
(Y)<\infty$. We begin with the

\begin{proof}
[\textbf{Proof of Theorem \ref{Thm_SameRetDistButDivergentHitDist}}%
]\textbf{(i)} We first recall that for any $B\in\mathcal{A}$ of finite
measure,
\begin{equation}
\frac{1}{n}\sum_{k=0}^{n-1}\int_{B}\widehat{T}^{k}1_{Y}\,d\mu\longrightarrow
0\text{ \quad as }n\rightarrow\infty\text{.}\label{Eq_LocalDecayOfMass}%
\end{equation}
Observe next that for every sequence $(m_{k})$ of integers $m_{k}\geq0$ the
sets $E_{k}^{\prime}:=T^{-m_{k}}E_{k}$, $k\geq1$, satisfy
\[
\mu_{E_{k}^{\prime}}(\gamma_{T}(\mu(E_{k}^{\prime}))\,\varphi_{E_{k}^{\prime}%
}\leq t)=\mu_{E_{k}}(\gamma_{T}(\mu(E_{k}))\,\varphi_{E_{k}}\leq t)\text{
\quad for }t\geq0\text{, }k\geq1\text{.}%
\]
Indeed, for $E\in\mathcal{A}$ with $0<\mu(E)<\infty$, and integers $m\geq0$,
it is immediate that
\begin{equation}
\varphi_{T^{-m}E}=\varphi_{E}\circ T^{m}\text{ \quad and \quad}\mu_{T^{-m}%
E}\circ T^{-m}=\mu_{E}\text{,}%
\end{equation}
and thus
\begin{equation}
\text{\textsf{law}}_{\mu_{T^{-m}E}}[\gamma(\mu(T^{-m}E))\,\varphi_{T^{-m}%
E}]=\text{\textsf{law}}_{\mu_{E}}[\gamma(\mu(E))\,\varphi_{E}]\text{.}%
\end{equation}
\textbf{(ii)} Fix any $k\geq1$ and set $B_{k}:=\{\varphi_{E_{k}}\leq
k/\gamma(\mu(E_{k}))\}=%
{\textstyle\bigcup_{j=1}^{\left\lfloor k/\gamma(\mu(E_{k}))\right\rfloor }}
T^{-j}E_{k}\in\mathcal{A}$, which has finite measure. If $m\geq0$ and
$E^{\prime}:=T^{-m}E_{k}$, then $\{\gamma(\mu(E^{\prime}))\varphi_{E^{\prime}%
}\leq k\}=T^{-m}B_{k}$ and hence
\[
\mu_{Y}(\gamma(\mu(E^{\prime}))\,\varphi_{E^{\prime}}\leq k)=\mu_{Y}%
(T^{-m}B_{k})=\mu(Y)^{-1}\int_{B_{k}}\widehat{T}^{m}1_{Y}\,d\mu\text{.}%
\]
But in view of (\ref{Eq_LocalDecayOfMass}) there is some $m=:m_{k}$ for which
$\int_{B_{k}}\widehat{T}^{m}1_{Y}\,d\mu<\mu(Y)/k$, and we take $E_{k}^{\prime
}:=T^{-m_{k}}E_{k}$.
\end{proof}%

\vspace{0.2cm}%
%

\noindent
\textbf{Arguments involving uniform sets.} Assume now that the system is
pointwise dual ergodic with $a_{T}\in\mathcal{R}_{\alpha}$ for some $\alpha
\in\lbrack0,1]$, and let $b_{T}$ be asymptotically inverse to $a_{T} $.
Suppose that $Y$ is a $u$-uniform set, w.l.o.g. with $\mu(Y)=1$, and that
$E_{k}\subseteq Y$, $k\geq1$, are sets of positive finite measure with
$\mu(E_{k})\rightarrow0$.

Our argument exploits the Ansatz of \cite{PSZ2} (which goes back to
\cite{DE}), and uses several auxiliary facts already mentioned or established
there. The following (which slightly generalizes Lemma 5.3 of \cite{PSZ2}) is
our starting point.

\begin{lemma}
[\textbf{Decomposition according to the last visit before time} $n$%
]\label{L_ED_basic_decomposition}Let $T$ be a c.e.m.p.t. on $(X,\mathcal{A}%
,\mu)$, and $B\in\mathcal{A}$ while $u\in\mathcal{D}(\mu)$. Then
\begin{equation}
\int_{\{\varphi_{B}\leq n\}}u\,d\mu=\sum_{l=1}^{n}\int_{B\cap\{\varphi
_{B}>n-l\}}\widehat{T}^{l}u\,d\mu\text{ \quad for }n\geq0\text{.}%
\label{Eq_ED_basic_decomposition}%
\end{equation}

\end{lemma}

\begin{proof}
Fix $n\geq0$ and decompose $\{\varphi_{B}\leq n\}$ according to the last
instant $l\in\{1,\ldots,n\}$ at which the orbit visits $B$ to get
\[
\{\varphi_{B}\leq n\}=%
{\textstyle\bigcup\nolimits_{l=1}^{n}}
T^{-l}(B\cap\{\varphi_{B}>n-l\})\text{ \quad(disjoint).}%
\]
Now measure these sets using the probability with density $u$.
\end{proof}%

\vspace{0.3cm}%

As our goal is to prove (\ref{Eq_AbstrDistrCgeI}) and
(\ref{Eq_AbstrDistrCgeIIIa}), we define
\begin{equation}
\mathsf{R}_{k}:=\mu(E_{k})\,a_{T}(\varphi_{E_{k}})\text{, \quad for }%
k\geq1\text{,}\label{Eq_DefRk}%
\end{equation}
and denote the relevant distribution functions by $\widetilde{G}_{k}$ and
$G_{k}$,
\begin{equation}
\widetilde{G}_{k}(t):=\mu_{E_{k}}(\mathsf{R}_{k}\leq t)\text{,\quad\ }%
G_{k}(t):=\mu_{Y}(\mathsf{R}_{k}\leq t)\text{\quad for }k\geq1\text{, }%
t\in\lbrack0,\infty)\text{.}%
\end{equation}
It is convenient to set, for $t\in\lbrack0,\infty)$ and $k\geq1$,
\begin{equation}
n_{k}^{[t]}:=b_{T}\left(  t/\mu(E_{k})\right)  \text{,}\label{Eq_DefNk}%
\end{equation}
and, for $l\in\{0,\ldots,n_{k}^{[t]}\}$, $\vartheta_{k,l}^{[t]}:=\mu
(E_{k})\cdot a_{T}(n_{k}^{[t]}-l)$. These allow us to represent some important
events, as
\begin{equation}
\left\{  \mathsf{R}_{k}>t\right\}  =\{\varphi_{E_{k}}>n_{k}^{[t]}\}\text{,
\quad while \quad}\{\mathsf{R}_{k}>\vartheta_{k,l}^{[t]}\}=\{\varphi_{E_{k}%
}>n_{k}^{[t]}-l\}\text{.}\label{Eq_Meaning_Nk_Thetak}%
\end{equation}
Note that, for any fixed $t$ and $k$, the function $l\mapsto\vartheta
_{k,l}^{[t]}$ is non-increasing.%

\vspace{0.3cm}%

We are now ready for the

\begin{proof}
\textbf{Proof of Theorem \ref{T_Ret_versus_Hit2}, case }$\alpha\in
(0,1]$\textbf{. (i)} note first that given $\rho\in\lbrack0,1)$ and a sequence
$(l_{k})_{k\geq1}$ in $(0,\infty)$,
\begin{equation}
\text{if \quad}l_{k}\sim\rho\cdot n_{k}^{[t]}\text{, \quad then \quad
}\vartheta_{k,l_{k}}^{[t]}\sim t\cdot(1-\rho)^{\alpha}\text{ \quad as
}k\rightarrow\infty\text{.}\label{Eq_AsyTheta}%
\end{equation}
Moreover, by the uniform set property of $Y$, if $0\leq c_{1}<c_{2}$, then
\begin{equation}
\sum_{j=c_{1}n}^{c_{2}n-1}\widehat{T}^{j}u\sim(c_{2}^{\alpha}-c_{1}^{\alpha
})\cdot a_{n}\qquad\text{as }n\rightarrow\infty\text{, uniformly mod }%
\mu\text{ on }Y\text{.}\label{Eq_DK_Sections}%
\end{equation}
%

\noindent
\textbf{(ii)} To exploit the fact that $Y$ is a $u$-uniform set, we consider
hitting-time distributions with respect to the probability with density $u$,
and set
\begin{equation}
G_{k}^{u}(t):=\int_{\{\mathsf{R}_{k}\leq t\}}u\,d\mu\text{, \quad for }%
k\geq1\text{, }t\in\lbrack0,\infty)\text{.}\label{Eq_DefGu}%
\end{equation}
In view of Remark \ref{Rem_Comments1} g), we see that
(\ref{Eq_AbstrDistrCgeIIIa}) is equivalent to
\begin{equation}
G_{k}^{u}(t)\Longrightarrow G(t)\text{ \quad as }k\rightarrow\infty
\text{,}\label{Eq_AbstrDistrCgeIIIau}%
\end{equation}
and for the rest of the proof we work with the latter condition.%

\vspace{0.1cm}%
%

\noindent
\textbf{(iii)} Assume that $\widetilde{G}_{k}\Rightarrow\widetilde{G}$, and
define $G$ via
\begin{equation}
G(t)=\alpha t\int_{0}^{1}[1-\widetilde{G}\left(  t\,(1-r)^{\alpha}\right)
]\,r^{\alpha-1}\,dr\text{ \quad for }t\in\lbrack0,\infty)\text{.}%
\label{Eq_ThisDefinesF}%
\end{equation}
We are going to prove that there is a dense subset $\mathcal{T}$ of
$[0,\infty)$ such that
\begin{equation}
G_{k}^{u}(t)\longrightarrow G(t)\text{ \quad for }t\in\mathcal{T}%
\text{,}\label{Eq_Convergence_for_nice_t}%
\end{equation}
It is then immediate that $G$ is a sub-probability distribution function on
$[0,\infty)$ for which $G_{k}^{u}\Longrightarrow G$.

To this end, let $\mathcal{T}$ be the set of those continuity points
$t\in(0,\infty)$ of $\widetilde{G}$ with the property that for all integers
$0\leq m\leq M$, the $t(1-\frac{m}{M})^{\alpha}$ also are continuity points of
$\widetilde{G}$.\ The complement of this set is only countable.

Henceforth, we fix some $t\in\mathcal{T}$, and abbreviate $n_{k}:=n_{k}^{[t]}$
and $\vartheta_{k,i}:=\vartheta_{k,i}^{[t]}$.%
\newline

Lemma \ref{L_ED_basic_decomposition}, with $B:=E_{k}$, and $n:=n_{k}$, gives
\begin{equation}
G_{k}^{u}(t)=\int_{\{\varphi_{E_{k}}\leq n_{k}\}}u\,d\mu=\sum_{l=1}^{n_{k}%
}\int_{E_{k}\cap\{\varphi_{E_{k}}>n_{k}-l\}}\widehat{T}^{l}u\,d\mu
\text{,}\label{Eq_FromDEDecomp}%
\end{equation}
and we are going to prove, for $k\rightarrow\infty$, that
\begin{equation}
\sum_{l=1}^{n_{k}}\int_{E_{k}\cap\{\varphi_{E_{k}}>n_{k}-l\}}\widehat{T}%
^{l}u\,d\mu\longrightarrow\alpha t\int_{0}^{1}[1-\widetilde{G}\left(
t\,(1-r)^{\alpha}\right)  ]\,r^{\alpha-1}\,dr\text{.}%
\label{Eq_Convergence_RHS}%
\end{equation}
%

\vspace{0.1cm}%
%

\noindent
\textbf{(iv)} Since $b_{T}\in\mathcal{R}_{1/\alpha}$, we have $n_{k}\sim
t^{1/\alpha}b_{T}(1/\mu(E_{k}))$ as $k\rightarrow\infty$. Fix some $M\geq1$,
and take any $\varepsilon\in(0,1)$. Decomposing the sum in
(\ref{Eq_Convergence_RHS}) into $M$ sections and recalling
(\ref{Eq_Meaning_Nk_Thetak}), we find that
\begin{align*}
\sum_{l=1}^{n_{k}}\int_{E_{k}\cap\{\varphi_{E_{k}}>n_{k}-l\}}\widehat{T}%
^{l}u\,d\mu & =\sum_{m=0}^{M-1}\sum_{l=\left\lfloor \frac{m}{M}n_{k}%
\right\rfloor +1}^{\left\lfloor \frac{m+1}{M}n_{k}\right\rfloor }\int
_{E_{k}\cap\{\mathsf{R}_{k}>\vartheta_{k,l}\}}\widehat{T}^{l}u\,d\mu\\
& \leq\sum_{m=0}^{M-1}\int_{E_{k}\cap\{\mathsf{R}_{k}>\vartheta
_{k,\left\lfloor \frac{m+1}{M}n_{k}\right\rfloor }\}}\sum_{l=\left\lfloor
\frac{m}{M}n_{k}\right\rfloor +1}^{\left\lfloor \frac{m+1}{M}n_{k}%
\right\rfloor }\widehat{T}^{l}u\,d\mu\text{,}%
\end{align*}
where the second step uses that, by monotonicity of $l\mapsto\vartheta_{k,l}$,
$\{\mathsf{R}_{k}>\vartheta_{k,l}\}\subseteq\{\mathsf{R}_{k}>\vartheta
_{k,\left\lfloor (m+1)n_{k}/M\right\rfloor }\}$ for $l\leq(m+1)n_{k}/M$. In
view of (\ref{Eq_DK_Sections}) and $a_{T}(n_{k})=t/\mu(E_{k})$ we have, for
$m\geq0$,
\[
\sum_{l=\left\lfloor \frac{m}{M}n_{k}\right\rfloor +1}^{\left\lfloor
\frac{m+1}{M}n_{k}\right\rfloor }\widehat{T}^{l}u\sim\left(  \left(
\frac{m+1}{M}\right)  ^{\alpha}-\left(  \frac{m}{M}\right)  ^{\alpha}\right)
\,\frac{t}{\mu(E_{k})}\qquad%
\begin{array}
[c]{c}%
\text{as }k\rightarrow\infty\text{,}\\
\text{uniformly mod }\mu\text{ on }Y\text{.}%
\end{array}
\]
Since $(\frac{m+1}{M})^{\alpha}-(\frac{m}{M})^{\alpha}\leq\alpha\frac{1}%
{M}(\frac{m}{M})^{\alpha-1}$ by the mean-value theorem, we thus get
\[
\sum_{l=1}^{n_{k}}\int_{E_{k}\cap\{\varphi_{E_{k}}>n_{k}-l\}}\widehat{T}%
^{l}u\,d\mu\leq e^{\varepsilon}\alpha\,t\sum_{m=1}^{M-1}\mu_{E_{k}}\left(
\mathsf{R}_{k}>\vartheta_{k,\left\lfloor \frac{m+1}{M}n_{k}\right\rfloor
}\right)  \left(  \frac{m}{M}\right)  ^{\alpha-1}\frac{1}{M}%
\]
for $k\geq K=K(M,\varepsilon)$. By our choice of $t$, $\widetilde{G}$ is
continuous at each $t(1-\frac{m+1}{M})^{\alpha}$, so that (\ref{Eq_AsyTheta})
ensures, as $k\rightarrow\infty$,
\begin{align*}
\mu_{E_{k}}\left(  \mathsf{R}_{k}>\vartheta_{k,\left\lfloor \frac{m+1}{M}%
n_{k}\right\rfloor }\right)   & =1-\widetilde{G}_{k}\left(  \vartheta
_{k,\left\lfloor \frac{m+1}{M}n_{k}\right\rfloor }\right) \\
& \longrightarrow1-\widetilde{G}\left(  t\left(  1-\frac{m+1}{M}\right)
^{\alpha}\right)  \text{.}%
\end{align*}
Combining this with the above, and letting $\varepsilon\searrow0$, we obtain
\[
\underset{k\rightarrow\infty}{\overline{\lim}}\sum_{l=1}^{n_{k}}\int
_{E_{k}\cap\{\varphi_{E_{k}}>n_{k}-l\}}\widehat{T}^{j}u\,d\mu\leq\alpha
\,t\sum_{m=1}^{M-1}\left[  1-\widetilde{G}\left(  t\left(  1-\frac{m+1}%
{M}\right)  ^{\alpha}\right)  \right]  \left(  \frac{m}{M}\right)  ^{\alpha
-1}\frac{1}{M}\text{.}%
\]
Now $M\rightarrow\infty$ yields
\[
\underset{k\rightarrow\infty}{\overline{\lim}}\sum_{l=1}^{n_{k}}\int
_{E_{k}\cap\{\varphi_{E_{k}}>n_{k}-l\}}\widehat{T}^{j}u\,d\mu\leq\alpha
t\int_{0}^{1}[1-\widetilde{G}\left(  t\,\left(  1-r\right)  ^{\alpha}\right)
]\,r^{\alpha-1}\,dr\text{.}%
\]
A parallel argument proves the corresponding lower estimate,
\[
\underset{k\rightarrow\infty}{\underline{\lim}}\sum_{l=1}^{n_{k}}\int
_{E_{k}\cap\{\varphi_{E_{k}}>n_{k}-l\}}\widehat{T}^{j}u\,d\mu\geq\alpha
t\int_{0}^{1}[1-\widetilde{G}\left(  t\,\left(  1-r\right)  ^{\alpha}\right)
]\,r^{\alpha-1}\,dr\text{,}%
\]
and hence our claim (\ref{Eq_Convergence_RHS}).%
\newline
%

\noindent
\textbf{(v)} Now assume that $G_{k}^{u}\Rightarrow G$. Our goal is to show
that $\widetilde{G}_{k}\Rightarrow\widetilde{G}$ with $\widetilde{G}$
satisfying (\ref{Eq_ThisDefinesF}). In view of the Helly selection theorem we
need only check that whenever $\widetilde{G}_{k_{i}}\Rightarrow\widetilde
{G}_{\ast}$ for some subsequence $k_{i}\nearrow\infty$ of indices, this limit
point $\widetilde{G}_{\ast}$ is indeed the unique sub-distribution function
satisfying (\ref{Eq_ThisDefinesF}).

However, if we apply the conclusion of step (i) above, we see that
$\widetilde{G}_{k_{i}}\Rightarrow\widetilde{G}_{\ast}$ entails $G_{k_{i}}%
^{u}\Rightarrow G_{\ast}$ with the pair $(G_{\ast},\widetilde{G}_{\ast})$
satisfying the desired integral equation. Since $G_{k}^{u}\Rightarrow G$ it is
clear that $G_{\ast}=G$, so that in fact $(G,\widetilde{G}_{\ast})$ satisfying
the integral equation. \
\end{proof}%

\vspace{0.1cm}%

A slight modification of the argument gives the

\begin{proof}
\textbf{Proof of Theorem \ref{T_Ret_versus_Hit2}, case }$\alpha=0$\textbf{.
(i)} As in the previous proof, we will work with the $G_{k}^{u}$ from
(\ref{Eq_DefGu}) rather than the $G_{k}$, and replace
(\ref{Eq_AbstrDistrCgeIIIa}) by (\ref{Eq_AbstrDistrCgeIIIau}).%

\vspace{0.1cm}%
%

\noindent
\textbf{(ii)} Assume that $\widetilde{G}_{k}\Rightarrow\widetilde{G}$. By a
subsequence-in-subsequence argument, we may assume that also $G_{k}%
^{u}\Rightarrow G$ for some $G\in\mathcal{F}$. We are going to prove that for
every point $t>0$ (henceforth fixed)\ at which both $G$ and $\widetilde{G}$
are continuous, we have
\begin{equation}
G_{k}^{u}(t)\longrightarrow t\,[1-\widetilde{G}\left(  t\right)  ]\text{ \quad
as }k\rightarrow\infty\text{,}\label{Eq_sjdhgfjekgfkerhfv}%
\end{equation}
whence $G(s)=s\,[1-\widetilde{G}\left(  s\right)  ]$, $s>0$. Abbreviate
$n_{k}:=n_{k}^{[t]}$.%
\newline

Observe that $\{\varphi_{E_{k}}>n_{k}/2\}=\{\mathsf{R}_{k}>\mu(E_{k}%
)\,a(n_{k}/2)\}$, where, due to slow variation of $a_{T}$, $\mu(E_{k}%
)\,a(n_{k}/2)\sim\mu(E_{k})\,a(n_{k})=t$ as $k\rightarrow\infty$. Since
$\widetilde{G}$ is continuous at $t$, we thus see that
\begin{align}
\mu_{E_{k}}(\varphi_{E_{k}}>n_{k}/2)  & =1-\widetilde{G}_{k}(\mu
(E_{k})\,a(n_{k}/2))\label{Eq_hdfgjashgfsagfjhgf}\\
& \longrightarrow1-\widetilde{G}(t)\text{ \quad as }k\rightarrow\infty
\text{.}\nonumber
\end{align}
%

\vspace{0.1cm}%
%

\noindent
\textbf{(iii)} By the uniform set property and $a(n_{k}/2)\sim a(n_{k})$ we
have%
\[
\sum_{l=n_{k}/2+1}^{n_{k}}\widehat{T}^{l}u=o\left(  a(n_{k})\right)
\qquad\text{as }k\rightarrow\infty\text{, uniformly mod }\mu\text{ on
}Y\text{,}%
\]
and hence, since $a(n_{k})=t/\mu(E_{k})$ and $E_{k}\subseteq Y$, we find that
\begin{align}
\sum_{l=n_{k}/2+1}^{n_{k}}\int_{E_{k}\cap\{\varphi_{E_{k}}>n_{k}-l\}}%
\widehat{T}^{l}u\,d\mu & \leq\int_{E_{k}}\sum_{l=n_{k}/2+1}^{n_{k}}\widehat
{T}^{l}u\,d\mu\label{Eq_KillTheLongTail}\\
& =\mu(E_{k})o\left(  a(n_{k})\right) \nonumber\\
& \longrightarrow0\text{ \qquad as }k\rightarrow\infty\text{.}\nonumber
\end{align}
On the other hand, since $\{\varphi_{E_{k}}>n_{k}-l\}\subseteq\{\varphi
_{E_{k}}>n_{k}/2\}$ for $l\leq n_{k}/2$, we can appeal to the uniform set
property and to (\ref{Eq_hdfgjashgfsagfjhgf}) to conclude that
\begin{align}
\sum_{l=1}^{n_{k}/2}\int_{E_{k}\cap\{\varphi_{E_{k}}>n_{k}-l\}}\widehat{T}%
^{l}u\,d\mu & \leq\int_{E_{k}\cap\{\varphi_{E_{k}}>n_{k}/2\}}\sum_{l=1}%
^{n_{k}/2}\widehat{T}^{l}u\,d\mu\nonumber\\
& \sim a(n_{k}/2)\,\mu(E_{k}\cap\{\varphi_{E_{k}}>n_{k}/2\})\nonumber\\
& \sim a(n_{k})\,\mu(E_{k}\cap\{\varphi_{E_{k}}>n_{k}/2\})\nonumber\\
& \sim t\,\mu_{E_{k}}(\varphi_{E_{k}}>n_{k}/2)\nonumber\\
& \longrightarrow t\,[1-\widetilde{G}(t)]\text{ \quad as }k\rightarrow
\infty\text{.}\label{Eq_EstimateTheShortBit1}%
\end{align}
In view of (\ref{Eq_FromDEDecomp}), (\ref{Eq_KillTheLongTail}) and
(\ref{Eq_EstimateTheShortBit1}) together give
\[
\underset{k\rightarrow\infty}{\overline{\lim}}G_{k}^{u}(t)\leq
t\,[1-\widetilde{G}(t)]\text{.}%
\]
%

\noindent
\textbf{(iv)} To also prove $\underline{\lim}_{k\rightarrow\infty}G_{k}%
^{u}(t)\geq t\,[1-\widetilde{G}(t)]$, and hence (\ref{Eq_sjdhgfjekgfkerhfv}),
we need only observe that by arguments similar to the above,
\begin{align}
G_{k}^{u}(t)=\sum_{l=1}^{n_{k}}\int_{E_{k}\cap\{\varphi_{E_{k}}>n_{k}%
-l\}}\widehat{T}^{l}u\,d\mu & \geq\int_{E_{k}\cap\{\varphi_{E_{k}}>n_{k}%
\}}\sum_{l=1}^{n_{k}}\widehat{T}^{l}u\,d\mu\nonumber\\
& \sim a(n_{k})\,\mu(E_{k}\cap\{\varphi_{E_{k}}>n_{k}\})\nonumber\\
& \sim t\,\mu_{E_{k}}(\varphi_{E_{k}}>n_{k})\nonumber\\
& \longrightarrow t\,[1-\widetilde{G}(t)]\text{ \quad as }k\rightarrow
\infty\text{.}%
\end{align}%
\noindent
\textbf{(v)} Conversely, if we start from the assumption that $G_{k}%
^{u}\Rightarrow G$, and want to show that $\widetilde{G}_{k}\Rightarrow
\widetilde{G}$ with $\widetilde{G}$ satisfying $G(s)=s\,[1-\widetilde
{G}\left(  s\right)  ]$, $s>0$, we can use the above by arguing as in the last
step of the proof for the $\alpha\in(0,1]$ case.%

\vspace{0.1cm}%
%

\noindent
\textbf{(vi)} Finally, we check that $G(t^{\prime})-G(t)\leq t^{\prime}-t$ for
$t^{\prime}>t>0$. Let $n_{k}^{\prime}:=n_{k}^{[t^{\prime}]}$ and use
(\ref{Eq_FromDEDecomp}) twice to see that
\[
G_{k}^{u}(t^{\prime})\leq G_{k}^{u}(t)+\int_{E_{k}}\sum_{l=n_{k}+1}%
^{n_{k}^{\prime}}\widehat{T}^{l}u\,d\mu\text{.}%
\]
But
\[
\int_{E_{k}}\sum_{l=n_{k}+1}^{n_{k}^{\prime}}\widehat{T}^{l}u\,d\mu\sim
\mu(E_{k})\left(  a_{T}(n_{k}^{\prime})-a_{T}(n_{k})\right)  =t^{\prime
}-t\text{ \quad as }k\rightarrow\infty\text{,}%
\]
and the desired estimate follows. This proves Lipschitz continuity of $G$.
\end{proof}%

\vspace{0.1cm}%

Now recall the following folklore principle (see e.g. Lemma 1 of \cite{BZ}).

\begin{lemma}
[\textbf{Regular variation preserves distributional convergence}%
]\label{L_RegVarDistortion}Assume that $R_{n}$ and $R$ are random variables
taking values in $(0,\infty),$ and that $\rho_{n}^{-1}R_{n}\Longrightarrow R$
for constants $\rho_{n}\rightarrow\infty$. If $B\ $is regularly varying of
index $\beta\neq0$, then
\begin{equation}
\frac{B(R_{n})}{B(\rho_{n})}\Longrightarrow R^{\beta}\text{.}%
\end{equation}

\end{lemma}%

\vspace{0.1cm}%

This easily leads to

\begin{proof}
[\textbf{Proof of Theorem \ref{T_Ret_versus_Hit}.}]Fix $\alpha\in(0,1]$.
Applying Lemma \ref{L_RegVarDistortion} proves that (\ref{Eq_AbstrDistrCgeII})
is equivalent to (\ref{Eq_AbstrDistrCgeI}) with $\widetilde{G}(t)=\widetilde
{F}(t^{1/\alpha})$, while (\ref{Eq_AbstrDistrCgeIII}) is equivalent to
(\ref{Eq_AbstrDistrCgeIIIa}) with $G(t)=F(t^{1/\alpha})$. This relation
between $(\widetilde{G},G)$ and $(\widetilde{F},F)$ turns
(\ref{Eq_MagicFormula2}) into
\[
F(t)=t^{\alpha}\int_{0}^{1}[1-\widetilde{F}\left(  t\,(1-r)\right)  ]\,\alpha
r^{\alpha-1}\,dr
\]
which, after an obvious change of variables, becomes (\ref{Eq_MagicFormula1}).
\end{proof}%

\vspace{0.1cm}%

We can now establish the main positive result of Section \ref{Sec_Scales}.

\begin{proof}
[\textbf{Proof of Theorem \ref{T_TightnessInsideDK}.}]\textbf{(i)} We first
show that $\gamma_{T}$ is a tight scale for return times in $Y$. Assume
otherwise, then there is some $\delta>0$ and a sequence $(E_{k})$ in
$Y\cap\mathcal{A}$, $\mu(E_{k})>0$ such that $\mu_{E_{k}}(\gamma_{T}(\mu
(E_{k}))\cdot\varphi_{E_{k}}>k)>\delta$ for $k\geq1$. In view of Proposition
\ref{P_NoTightnesssLargeSets} b), this sequence must satisfy $\mu
(E_{k})\rightarrow0$.

Now Helly's selection theorem guarantees that there exist indices
$k_{j}\nearrow\infty$ s.t.
\[
\mu_{E_{k_{j}}}(\gamma_{T}(\mu(E_{k_{j}}))\,\varphi_{E_{k_{j}}}\leq
t)\Longrightarrow\widetilde{F}(t)\text{ \quad as }j\rightarrow\infty\text{,}%
\]
for some $\widetilde{F}\in\mathcal{F}$. By our choice of $(E_{k})$, the latter
satisfies $\sup_{t\in\lbrack0,\infty)}\widetilde{F}(t)<1-\delta$. But this
contradicts Remark \ref{Rem_Comments1} c).%
\newline
%

\noindent
\textbf{(ii)} To prove that $\gamma_{T}$ is a nontrivial scale for return
times in $Y$, we will construct a sequence $(E_{k})$ in $Y\cap\mathcal{A}$
with $0<\mu(E_{k})\rightarrow0$ for which there are $t^{\ast},k^{\ast}>0$
s.t.
\begin{equation}
\mu_{Y}(\gamma_{T}(\mu(E_{k}))\,\varphi_{E_{k}}\leq t^{\ast})>1/3\text{ \quad
for }k\geq k^{\ast}\text{.}\label{Eq_cbxvuzsfguazcuav}%
\end{equation}
Helly's selection theorem then provides us with $k_{j}\nearrow\infty$ and
$F,\widetilde{F}\in\mathcal{F}$ such that $\mu_{Y}(\gamma_{T}(\mu(E_{k_{j}%
}))\,\varphi_{E_{k_{j}}}\leq t)\Longrightarrow F(t)$ and $\mu_{E_{k_{j}}%
}(\gamma_{T}(\mu(E_{k_{j}}))\,\varphi_{E_{k_{j}}}\leq t)\Longrightarrow
\widetilde{F}(t)$ as $j\rightarrow\infty$. Due to (\ref{Eq_cbxvuzsfguazcuav}),
it is clear that $F(t^{\ast})\geq1/3$ (use the Portmanteau Theorem). In view
of (\ref{Eq_MagicFormula1}) from Theorem \ref{T_Ret_versus_Hit}, this implies
that there is some $s^{\ast}\in(0,t^{\ast})$ for which $1-\widetilde
{F}(s^{\ast})>0$. But then
\begin{equation}
\underline{\lim}_{j\rightarrow\infty}\,\mu_{E_{k_{j}}}(\gamma_{T}(\mu
(E_{k_{j}}))\,\varphi_{E_{k_{j}}}>s^{\ast})>0\text{,}%
\end{equation}
confirming that $\gamma_{T}$ is a nontrivial scale for return times in $Y$.%
\newline
%

\noindent
\textbf{(iii)} To find $(E_{k})$ satisfying (\ref{Eq_cbxvuzsfguazcuav}), we
combine two results regarding the first-return map $T_{Y}:Y\rightarrow Y$. The
first of these only uses the fact that $T_{Y}$ is an ergodic m.p.t. on the
probability space $(Y,Y\cap\mathcal{A},\mu_{Y})$, and that this space has no
atoms. To check the latter property, note that a space with atoms supporting
an ergodic m.p.t. must be purely atomic with a finite number of atoms. But if
the measure is supported on finitely many atoms, the space cannot support a
non-integrable almost surely finite function $\varphi_{Y}$.

Due to these basic properties, $T_{Y}$ admits a sequence $(E_{k})$ in
$Y\cap\mathcal{A}$ for which $\mu_{Y}(E_{k})\rightarrow0$ and
\begin{equation}
\mu_{Y}(\mu_{Y}(E_{k})\,\varphi_{E_{k}}^{Y}\leq t)\Longrightarrow
1-e^{-t}\text{ \quad as }k\rightarrow\infty\text{,}%
\end{equation}
see \cite{KL}. In particular, fixing some $\theta>0$ so large that
$1-e^{-\theta}>2/3$, we see that the sets $V_{k}:=Y\cap\{\varphi_{E_{k}}%
^{Y}\leq\theta/\mu_{Y}(E_{k})\}$, $k\geq1$, satisfy $\mu_{Y}(V_{k})>2/3$ for
$k\geq k^{\prime}$.%
\newline
%

\noindent
\textbf{(iv)} The second ingredient is Aaronson's Darling-Kac limit theorem
(Corollary 3.7.3 of \cite{A0}, see also \cite{A1}), which applies under the
present assumptions. This result is equivalent to a stable limit theorem for
the return time function $\varphi_{Y}$ under $T_{Y}$. The latter result
asserts that
\begin{equation}
\mu_{Y}\left(  \frac{1}{b_{T}(m)}%
{\textstyle\sum_{j=0}^{m-1}}
\varphi_{Y}\circ T_{Y}^{j}\leq t\right)  \Longrightarrow\Pr[\Gamma
(1+\alpha)^{-1/\alpha}\,\mathcal{G}_{\alpha}\leq t]\text{ \quad as
}k\rightarrow\infty\text{.}%
\end{equation}
Pick $\rho>0$ so large that $\Pr[\Gamma(1+\alpha)^{-1/\alpha}\,\mathcal{G}%
_{\alpha}\leq\rho]>2/3$. Then there is some $m^{\prime}$ such that the sets
$W_{m}:=Y\cap\{b_{T}(m)^{-1}%
{\textstyle\sum_{j=0}^{m-1}}
\varphi_{Y}\circ T_{Y}^{j}\leq\rho\}$ satisfy $\mu_{Y}(W_{m})>2/3$ for $m\geq
m^{\prime}$.%
\newline
%

\noindent
\textbf{(v)} Since $b_{T}\in\mathcal{R}_{1/\alpha}$ and $\mu(E_{k}%
)\rightarrow0$, we have $b_{T}(\left\lfloor \theta/\mu_{Y}(E_{k})\right\rfloor
)/b_{T}(1/\mu_{Y}(E_{k}))<2\theta^{1/\alpha}$ for $k\geq k^{\prime\prime}$.
Recalling (\ref{Eq_sdhfgjsdhgfjshdfjsafjlsagfhfgjsdhfjdshf}) we therefore see
that
\begin{align*}
\gamma_{T}(\mu(E_{k}))\,\varphi_{E_{k}}  & \leq\gamma_{T}(\mu(E_{k}%
))\,\sum_{j=0}^{\left\lfloor \theta/\mu_{Y}(E_{k})\right\rfloor -1}\varphi
_{Y}\circ T_{Y}^{j}\text{ \quad on }V_{k}\\
& =\frac{b_{T}(\left\lfloor \theta/\mu_{Y}(E_{k})\right\rfloor )}{b_{T}%
(1/\mu_{Y}(E_{k}))}\,\frac{1}{b_{T}(\left\lfloor \theta/\mu_{Y}(E_{k}%
)\right\rfloor )}\sum_{j=0}^{\left\lfloor \theta/\mu_{Y}(E_{k})\right\rfloor
-1}\varphi_{Y}\circ T_{Y}^{j}\text{ }\\
& <2\theta^{1/\alpha}\rho=:t^{\ast}\text{ \ \quad on }W_{\left\lfloor
\theta/\mu_{Y}(E_{k})\right\rfloor }\text{ for }k\geq k^{\prime\prime}\text{.}%
\end{align*}
Hence,
\[
\mu_{Y}(\gamma_{T}(\mu(E_{k}))\,\varphi_{E_{k}}\leq t^{\ast})\geq\mu_{Y}%
(V_{k}\cap W_{\left\lfloor \theta/\mu_{Y}(E_{k})\right\rfloor })\text{ \quad
for }k\geq k^{\prime\prime}\text{.}%
\]
But there is some $k^{\prime\prime\prime}$ such that $\left\lfloor \theta
/\mu_{Y}(E_{k})\right\rfloor \geq m^{\prime}$ for $k\geq k^{\prime\prime
\prime}$, and then
\[
\mu_{Y}(V_{k}\cap W_{\left\lfloor \theta/\mu_{Y}(E_{k})\right\rfloor
})>1/3\text{ \quad for }k\geq k^{\ast}:=k^{\prime}\vee k^{\prime\prime}\vee
k^{\prime\prime\prime}\text{,}%
\]
which shows that $(E_{k})$ satisfies (\ref{Eq_cbxvuzsfguazcuav}).
\end{proof}%

\vspace{0.1cm}%

Finally, we turn to the

\begin{proof}
[\textbf{Proof of Theorem \ref{P_NonlinTightScale}}]\textbf{(i)} As a
consequence of Proposition \ref{P_NoTightnesssLargeSets} b) we see that
whenever $a$ is a scaling function, then for ever $Y\in\mathcal{A}$ with
$0<\mu(Y)<\infty$ and any $\eta>0$, $\left\{  \text{\textsf{law}}_{\mu_{E}%
}[a(\varphi_{E})]:E\in Y\cap\mathcal{A}\text{, }\mu(E)\geq\eta\right\}  $ is
tight.\newline\newline\textbf{(ii)} Now take $Y$ and $a=a_{T}$ as in the
statement of the proposition. Assume for a contradiction that
(\ref{Eq_NonlinTightScale}) fails. Then there is some $\delta>0$ and a
sequence $(E_{k})$ in $Y\cap\mathcal{A}$ such that $\mu_{E_{k}}[\mu
(E_{k})\,a_{T}(\varphi_{E_{k}})>k]>\delta$ for $k\geq1$. By step (i) we
necessarily have $\mu(E_{k})\rightarrow0$. The Helly selection theorem
provides us with $k_{j}\nearrow\infty$ and $\widetilde{G}\in\mathcal{F}$ for
which
\begin{equation}
\mu_{E_{k_{j}}}(\mu(E_{k_{j}})\,a_{T}(\varphi_{E_{k_{j}}})\leq
t)\Longrightarrow\widetilde{G}(t)\text{ \quad as }j\rightarrow\infty
\text{.}\label{Eq_sysyasyas}%
\end{equation}
According to Theorem \ref{T_Ret_versus_Hit2} and Remark \ref{Rem_Comments2}%
\ b), however, $\widetilde{G}$ has to be a probability distribution function
on $[0,\infty)$. This contradicts our choice of $(E_{k})$.\newline%
\newline\textbf{(iii)} To show that this nonlinear scaling is nontrivial,
consider $E_{k}^{\ast}:=Y\cap\{\varphi_{Y}>k\}$, $k\geq1$. Then, $\mu
(E_{k}^{\ast})=\mu(Y)q_{k}(Y)$, and $\varphi_{E_{k}^{\ast}}\geq\varphi_{Y}>k$
on $E_{k}^{\ast}$. Consequently,
\begin{equation}
\mu(E_{k}^{\ast})\,a_{T}(\varphi_{E_{k}^{\ast}})>\mu(Y)q_{k}(Y)a_{T}(k)\text{
\quad on }E_{k}^{\ast}\text{.}\label{Eq_nvbnxvbnb}%
\end{equation}
But in the $\alpha=0$ case, the asymptotic renewal equation
(\ref{Eq_AsyRenewalEqn}) and the (monotone density part of) Karamata's
Tauberian Theorem yield%
\begin{equation}
a_{T}(k)\sim\frac{k}{w_{k}(Y)}\sim\frac{1}{\mu(Y)q_{k}(Y)}\text{ \quad as
}k\rightarrow\infty\text{.}\label{Eq_nbxcvnysvcsydvcjgd}%
\end{equation}
Together with (\ref{Eq_nvbnxvbnb}) this shows that any limit point
$\widetilde{G}$ as in (\ref{Eq_sysyasyas})\ of the (nonlinearly rescaled)
return-time distributons, vanishes on $[0,1)$. In view of Theorem
\ref{T_Ret_versus_Hit2}, this means that $\mu_{Y}(\mu(E_{k_{j}}^{\ast}%
)\,a_{T}(\varphi_{E_{k_{j}}^{\ast}})\leq t)\Longrightarrow G(t)$ with $G(t)=t$
for $t\in\lbrack0,1]$, and hence (since $G\in\mathcal{F}$) that $G(t)=t\wedge
1$ for $t\geq0$. Appealing to (\ref{Eq_MagicFormula2}) once again, we obtain
the explicit form (\ref{Eq_TheExtremeAlpha0Limit}) of $\widetilde{G}$.

This uniqueness of the limit point $\widetilde{G}$, together with Helly's
theorem implies convergence along the full sequence.
\end{proof}%

\vspace{0.1cm}%

\section{Proving convergence to $H_{\alpha}$ in uniform
sets\label{Sec_LimitThmAlphaPositive}}%

\noindent
\textbf{Characterizing convergence to }$H_{\alpha}$\textbf{.} In the case of
finite measure preserving systems, Theorem HLV is not only of interest in its
own right, but it is also the basis of a method for proving convergence to the
exponential distribution (see \cite{HSV}). Indeed, it is clear that the only
$F\in\mathcal{F}$ which satisfies
\[
F(t)=\int_{0}^{t}[1-F(s)]\,ds\text{ \quad for }t\geq0\text{,}%
\]
is $F=H_{1}$, where $H_{1}(t):=1-e^{-t}$, $t>0$. The most prominent limit law
is thus characterized as the unique distribution which can appear both as
return- and as hitting-time limit for the same sequence of sets. In view of
this and the Helly selection principle, one can prove convergence to
$\mathcal{E}$ of both return- and hitting-time distributions by showing that
the two types of distributions are asymptotically the same.

Here, we obtain (with hardly any effort) a result which allows a parallel
approach to proving convergence to $H_{\alpha}$ inside uniform sets in case
$a_{T}\in\mathcal{R}_{\alpha}$, $\alpha\in(0,1]$. The following is contained
in Lemma 7 of \cite{PSZ}\footnote{Note that the parameter $\alpha$ appearing
in Lemma 7 of \cite{PSZ} is not the same as our $\alpha$. In the notation of
the present paper it equals $1/(1-\alpha)$.\newline}.

\begin{lemma}
[\textbf{Characterization of} $H_{\alpha}$]\label{L_CharHalpha}For every
$\alpha\in(0,1]$ the distribution function $F=H_{\alpha}$ of $\mathcal{H}%
_{\alpha}$ is the unique element of $\mathcal{F}$ which satisfies%
\begin{equation}
F(t)=\int_{0}^{t}[1-F(s)]\,\alpha\left(  t-s\right)  ^{\alpha-1}ds\text{ \quad
for }t\geq0\text{.}%
\end{equation}

\end{lemma}

As a consequence, we obtain

\begin{theorem}
[\textbf{Characterizing convergence to }$H_{\alpha}$ \textbf{in uniform sets}%
]\label{T_Cge_to_H_alpha}Let $T$ be a c.e.m.p.t. on $(X,\mathcal{A},\mu)$,
$\mu(X)=\infty$. Assume that $Y\in\mathcal{A}$ is a uniform set and that
$a_{T}\in\mathcal{R}_{\alpha}$ for some $\alpha\in(0,1]$.

Suppose that $E_{k}\subseteq Y$, $k\geq1$, are sets of positive measure with
$\mu(E_{k})\rightarrow0$. Then the normalized return-time distributions of the
$E_{k}$ converge to $H_{\alpha}$,
\begin{equation}
\mu_{E_{k}}(\gamma_{T}(\mu(E_{k}))\,\varphi_{E_{k}}\leq t)\Longrightarrow
H_{\alpha}(t)\text{ \quad as }k\rightarrow\infty\text{,}%
\label{Eq_Ret_Cge_H_alpha}%
\end{equation}
if and only if the normalized hitting-time distributions converge to
$H_{\alpha}$,
\begin{equation}
\mu_{Y}(\gamma_{T}(\mu(E_{k}))\,\varphi_{E_{k}}\leq t)\Longrightarrow
H_{\alpha}(t)\text{ \quad as }k\rightarrow\infty\text{,}%
\label{Eq_Hit_Cge_H_alpha}%
\end{equation}
if and only if for a dense set of points $t\ $in $(0,\infty)$,
\begin{equation}
\mu_{E_{k}}(\gamma_{T}(\mu(E_{k}))\,\varphi_{E_{k}}\leq t)-\mu_{Y}(\gamma
_{T}(\mu(E_{k}))\,\varphi_{E_{k}}\leq t)\longrightarrow0\text{ \quad as
}k\rightarrow\infty\text{.}\label{Eq_Same_Law_For_Ret_and_Hit}%
\end{equation}

\end{theorem}%

\vspace{0.2cm}%

\begin{proof}
By Theorem \ref{T_Ret_versus_Hit} it is clear that (\ref{Eq_Ret_Cge_H_alpha})
is equivalent to (\ref{Eq_Hit_Cge_H_alpha}). Trivially, either of these
statements therefore implies (\ref{Eq_Same_Law_For_Ret_and_Hit}).

To prove the converse, start from (\ref{Eq_Same_Law_For_Ret_and_Hit}), and
assume for a contradiction that, say, (\ref{Eq_Ret_Cge_H_alpha}) fails, so
that by Helly's selection principle there is a subsequence $k_{j}%
\nearrow\infty$ of indices and some $\widetilde{F}\in\mathcal{F}$,
$\widetilde{F}\neq H_{\alpha}$, such that (\ref{Eq_AbstrDistrCgeII}) holds
along that subsequence. By Theorem \ref{T_Ret_versus_Hit}, so does
(\ref{Eq_AbstrDistrCgeIII}), where $F$ and $\widetilde{F}$ are related by
(\ref{Eq_MagicFormula1}). But (\ref{Eq_Same_Law_For_Ret_and_Hit}) ensures that
$\widetilde{F}=F$, which in view of Lemma \ref{L_CharHalpha} contradicts
$\widetilde{F}\neq H_{\alpha}$.

Exactly the same argument works if we assume that (\ref{Eq_Hit_Cge_H_alpha}) fails.
\end{proof}%

\vspace{0.2cm}%
%

\noindent
\textbf{Sufficient conditions for convergence to }$H_{\alpha}$\textbf{.} It is
natural to review the abstract distributional limit theorem of \cite{PSZ2},
which gives sufficient conditions for convergence to $\mathcal{H}_{\alpha} $,
in the light of the preceding result. We establish an improved limit theorem
which is similar in spirit, but uses easier assumptions.

The key idea goes back to \cite{Z7}, \cite{Z6}: Ergodicity ensures that the
asymptotics of distributions of variables $R_{k}$ which are asymptotically
invariant in measure does not depend on the choice of the initial density, see
Lemma \ref{L_StrongDCge}. In fact, there is uniform control over important
quantities as long as all densities involved belong to a sufficiently small
family. We shall use the following (see statement (3.1) of \cite{Z6}).

\begin{remark}
[\textbf{Dual ergodic sums for compact sets of densities}]\label{R_UnifMETH}%
Let $T$ be a c.e.m.p.t. on $(X,\mathcal{A},\mu)$, and $\mathcal{U}\subseteq
L_{1}(\mu)$ a (strongly) compact set of probablity densities. Then
\begin{equation}
\left\Vert \frac{1}{M}%
{\textstyle\sum_{j=0}^{M-1}}
\widehat{T}^{j}u-\frac{1}{M}%
{\textstyle\sum_{j=0}^{M-1}}
\widehat{T}^{j}u^{\ast}\right\Vert _{1}\longrightarrow0\text{ \quad}%
\begin{array}
[c]{c}%
\text{as }M\rightarrow\infty\text{,}\\
\text{uniformly in }u,u^{\ast}\in\mathcal{U}\text{.}%
\end{array}
\label{Eq_uououououo}%
\end{equation}

\end{remark}

This will entail a uniform (in the initial density) version of Lemma
\ref{L_StrongDCge}. We can use this to replace the fixed measure $\mu_{Y}$ in
the criterion (\ref{Eq_Same_Law_For_Ret_and_Hit}) of the previous theorem by
any sequence of probabilities $Q_{k}$ as long as their respective densities
$v_{k}$ stay in some definite compact set $\mathcal{U}$. Suppose this holds
for the particular case where each $Q_{k}$ is just a push-forward of the
conditional measure $\mu_{E_{k}}$ by a comparatively small (condition
(\ref{Eq_SmallDelay}) below) number $\tau$ of steps. Then, by virtue of
(\ref{Eq_uououououo}) those will exhibit the same asymptotics as the
$\mu_{E_{k}} $ provided that the exceptional set $\{\varphi_{E_{k}}<\tau\}$ is
not too large (condition (\ref{Eq_BoringBeforeTau})).

This is the core of the argument to follow. To get a flexible result, we allow
for delays $\tau$ which depend on the point. Given a measurable $\tau
:E\rightarrow\mathbb{N}_{0}$ we can define $T^{\tau}:E\rightarrow X$ by
$T^{\tau}x:=T^{\tau(x)}x$, which gives a null-preserving map. For $u\in
L_{1}(\mu)$ supported on $E$, let
\begin{equation}
\widehat{T^{\tau}}u:=%
{\textstyle\sum\nolimits_{j\geq0}}
\widehat{T^{j}}(1_{E\cap\{\tau=j\}}u)\text{,}\label{Eq_vcycvyvcvycvshgdv}%
\end{equation}
then $\widehat{T^{\tau}}$ is the transfer operator of $T^{\tau}$, describing
the push-forward of measures by $T^{\tau}$ on the level of densities,
\begin{equation}
\int_{X}f\cdot\,\widehat{T^{\tau}}u\,d\mu=\int_{E}(f\circ T^{\tau})\cdot
u\,d\mu\text{ \quad for }u\in L_{1}(\mu)\text{ and }f\in L_{\infty}%
(\mu)\text{.}%
\end{equation}
Using such an auxiliary map $T^{\tau}$, we can formulate the advertised result.

\begin{theorem}
[\textbf{Sufficient conditions for convergence to }$H_{\alpha}$]%
\label{T_SuffForCgeToHAlpha}Let $T$ be a c.e.m.p.t. on $(X,\mathcal{A},\mu)$,
$\mu(X)=\infty$, pointwise dual ergodic with $a_{T}\in\mathcal{R}_{\alpha}$
for some $\alpha\in(0,1]$. Suppose that $Y$ is a uniform set, and that
$E_{k}\subseteq Y$, $k\geq1$, are sets of positive measure with $\mu
(E_{k})\rightarrow0$.

Assume that there are measurable functions $\tau_{k}:E_{k}\rightarrow
\mathbb{N}_{0}$ and an $L_{1}(\mu)$-compact set $\mathcal{U}\subseteq
\mathcal{D}(\mu)$ of probability densities for which
\begin{equation}
\gamma_{T}(\mu(E_{k}))\,\tau_{k}\overset{\mu_{E_{k}}}{\longrightarrow}0\text{
\quad as }k\rightarrow\infty\text{,}\label{Eq_SmallDelay}%
\end{equation}
(in that $\mu_{E_{k}}[\gamma_{T}(\mu(E_{k}))\,\tau_{k}>\varepsilon
]\rightarrow0$ for every $\varepsilon>0$), and
\begin{equation}
\mu_{E_{k}}(\varphi_{E_{k}}<\tau_{k})\longrightarrow0\text{ \quad as
}k\rightarrow\infty\text{,}\label{Eq_BoringBeforeTau}%
\end{equation}
while
\begin{equation}
\widehat{T^{\tau_{k}}}(\mu(E_{k})^{-1}1_{E_{k}})\in\mathcal{U}\text{ \quad for
}k\geq1\text{.}\label{Eq_TheCompactnessAssm}%
\end{equation}
Then the return-time distributions of the $E_{k}$ converge to $H_{\alpha}$,
\begin{equation}
\mu_{E_{k}}(\gamma_{T}(\mu(E_{k}))\,\varphi_{E_{k}}\leq t)\Longrightarrow
H_{\alpha}(t)\text{ \quad as }k\rightarrow\infty\text{,}%
\label{Eq_jsdfsahdfasd}%
\end{equation}
and so do the hitting-time distributions,
\begin{equation}
\mu_{Y}(\gamma_{T}(\mu(E_{k}))\,\varphi_{E_{k}}\leq t)\Longrightarrow
H_{\alpha}(t)\text{ \quad as }k\rightarrow\infty\text{.}%
\label{Eq_bvczwtefuwzguf}%
\end{equation}

\end{theorem}%

\vspace{0.2cm}%

\begin{remark}
[\textbf{This result improves Theorem 4.1 of \cite{PSZ2} in that ...}%
]\label{Rem_WhyBetter}Our theorem is more general and more flexible than its
predecessor, Theorem 4.1 of \cite{PSZ2}. While we do not exploit this in the
present article, we expect that the following points will be useful in
extending the results on interval maps presented in that paper to other
classes of concrete systems and of asymptotically rare sequences.
\newline\textbf{a)} The condition (\ref{Eq_SmallDelay}) on the order of the
delay times $\tau_{k}$ is more general and more natural that the condition
$\mu(E_{k})\,\tau_{k}\longrightarrow^{\mu_{E_{k}}}0$ used in \cite{PSZ2}%
.\newline\textbf{b)} In \cite{PSZ2} the $\tau_{k}$ were required to be of a
more specific type. \newline\textbf{c)} To control the densities in
(\ref{Eq_TheCompactnessAssm}), we merely assume that $\mathcal{U}$ is compact
in $L_{1}(\mu)$, a property of $\mathcal{U}$ which does not involve the
dynamics. In contrast, \cite{PSZ2} required $Y$ to be a $\mathcal{U}%
$\emph{-uniform set}, which is a dynamical condition (and thus harder to
check): it means that $\mathcal{U}\subseteq\mathcal{D}(\mu)$ is a class of
densities such that the $L_{\infty}(\mu)$-convergence asserted in
(\ref{Eq_Def_UnifSet}) holds uniformly in $u\in\mathcal{U}$, that is,
\begin{equation}
\sum_{k=0}^{n-1}\widehat{T}^{k}u\sim a_{n}\qquad%
\begin{array}
[c]{c}%
\text{as }n\rightarrow\infty\text{, uniformly mod }\mu\text{ on }Y\text{,}\\
\text{and uniformly in }u\in\mathcal{U}\text{.}%
\end{array}
\label{Eq_UunifSetRatioVersion}%
\end{equation}
\newline
\end{remark}

The following proof of the theorem seems more transparent than that of Theorem
4.1 of \cite{PSZ2}, its basic strategy being parallel to a standard argument
from the finite-measure case.%

\vspace{0.2cm}%

\begin{proof}
\textbf{(i)} Equivalence of (\ref{Eq_jsdfsahdfasd}) and
(\ref{Eq_bvczwtefuwzguf}) is clear from Theorem \ref{T_Cge_to_H_alpha}, and we
prove them by validating condition (\ref{Eq_Same_Law_For_Ret_and_Hit}) of that
result. Let $R_{k}:=\gamma_{T}(\mu(E_{k}))\,\varphi_{E_{k}}$ and
$F_{k}(t):=\mu_{Y}(R_{k}\leq t)$, $\widetilde{F}_{k}(t):=\mu_{E_{k}}(R_{k}\leq
t)$, $t\geq0$. We prove (\ref{Eq_Same_Law_For_Ret_and_Hit}) by showing that
\begin{equation}
F_{k}-\widehat{F}_{k}\rightarrow0\text{\quad and\quad}\widehat{F}%
_{k}-\widetilde{F}_{k}\rightarrow0\text{ \quad}\lambda\text{-a.e. on
}[0,\infty)\text{ as }k\rightarrow\infty\text{,}%
\end{equation}
where $\widehat{F}_{k}$ denotes the distribution function of $R_{k}$ with
respect to the measure with density $v_{k}:=\widehat{T^{\tau_{k}}}(\mu
(E_{k})^{-1}1_{E_{k}})$, so that $\widehat{F}_{k}(t):=\int_{\{R_{k}\leq
t\}}v_{k}\,d\mu$.%
\newline
%

\noindent
\textbf{(ii)} Recall that by definition (\ref{Eq_vcycvyvcvycvshgdv}) of
$\widehat{T^{\tau}}$ we have $\widehat{F}_{k}(t)=\mu_{E_{k}}(R_{k}\circ
T^{\tau_{k}}\leq t)$. To first check that for $\lambda$-a.e. $t\in
\lbrack0,\infty)$,
\begin{equation}
\mu_{E_{k}}(R_{k}\leq t)-\mu_{E_{k}}(R_{k}\circ T^{\tau_{k}}\leq
t)\longrightarrow0\text{ \quad as }k\rightarrow\infty\text{.}%
\label{Eq_lhjkkhjgkjgklgkjjgngklmjk}%
\end{equation}
we need only observe that (generalizing (\ref{Eq_nbmcxnbhdfvbmchvvrhf})), for
any measurable $E$ and $\tau\geq0$,
\begin{equation}
\varphi_{E}=\varphi_{E}\circ T^{\tau}+\tau\text{ \quad on }\{\varphi_{E}%
>\tau\}\text{,}\label{Eq_xvcygvdshgvchasvaaaaaa}%
\end{equation}
and hence
\begin{equation}
R_{k}=R_{k}\circ T^{\tau_{k}}+\gamma_{T}(\mu(E_{k}))\,\tau_{k}\text{ \quad on
}\{\varphi_{E_{k}}>\tau_{k}\}\text{.}%
\end{equation}
Given this, (\ref{Eq_lhjkkhjgkjgklgkjjgngklmjk}) is an easy consequence of
(\ref{Eq_SmallDelay}) and (\ref{Eq_BoringBeforeTau}). This shows that
$\widehat{F}_{k}-\widetilde{F}_{k}\rightarrow0$ holds\ $\lambda$-a.e. on
$[0,\infty)$.%
\newline
%

\noindent
\textbf{(iii)} To also prove that $F_{k}-\widehat{F}_{k}\rightarrow0$ a.e. on
$[0,\infty)$, note first that (by the general definition of distributional
convergence) this statement is equivalent to saying that for every bounded
Lipschitz function $\psi:\mathbb{R}\rightarrow\mathbb{R}$ one has
\begin{equation}
\int(\psi\circ R_{k})\,1_{Y}\,d\mu-\int(\psi\circ R_{k})\,v_{k}\,d\mu
\longrightarrow0\text{ \quad as }k\rightarrow\infty\text{,}%
\end{equation}
which, by (\ref{Eq_TheCompactnessAssm}), will follow once we show that for any
such $\psi$,
\begin{equation}
\int(\psi\circ R_{k})\,u\,d\mu-\int(\psi\circ R_{k})\,u^{\ast}\,d\mu
\longrightarrow0\text{ \quad}%
\begin{array}
[c]{c}%
\text{as }k\rightarrow\infty\text{,}\\
\text{uniformly in }u,u^{\ast}\in\mathcal{U}%
\end{array}
\label{Eq_TheUnifCgeTrick}%
\end{equation}
(assume w.l.o.g. that $1_{Y}\in\mathcal{U}$). To this end, we now fix $\psi$
and any $\varepsilon>0$.%
\newline
%

\noindent
\textbf{(iv)} In view of Remark \ref{R_UnifMETH} there is some $M\geq1$ such
that
\begin{align}
\left\vert \int(\psi\circ R_{k})\,\frac{1}{M}%
{\textstyle\sum_{j=0}^{M-1}}
\widehat{T}^{j}(u-u^{\ast})\,d\mu\right\vert  & \leq\sup\left\vert
\psi\right\vert \,\left\Vert \frac{1}{M}%
{\textstyle\sum_{j=0}^{M-1}}
\widehat{T}^{j}(u-u^{\ast})\right\Vert _{1}\nonumber\\
& <\frac{\varepsilon}{3}\text{ \quad for all }u,u^{\ast}\in\mathcal{U}%
\text{.}\label{Eq_Smear}%
\end{align}
On the other hand, for any $j\in\{0,\ldots,M-1\}$ and $u\in\mathcal{U}$, we
find recalling (\ref{Eq_xvcygvdshgvchasvaaaaaa}),%
\begin{align*}
\left\vert \int(\psi\circ R_{k})\,(u-\widehat{T}^{j}u)\,d\mu\right\vert  &
\leq\int\left\vert \psi\circ R_{k}-\psi\circ R_{k}\circ T^{j}\right\vert
\,u\,d\mu\\
& \leq2\sup\left\vert \psi\right\vert \,\int_{\{\varphi_{E_{k}}\leq
j\}}u\,d\mu+\mathrm{Lip}(\psi)\,\gamma_{T}(\mu(E_{k}))\,j\text{,}%
\end{align*}
so that
\begin{align*}
\left\vert \int(\psi\circ R_{k})\,\left(  u-\frac{1}{M}%
{\textstyle\sum_{j=0}^{M-1}}
\widehat{T}^{j}u\right)  \,d\mu\right\vert  & \leq2\sup\left\vert
\psi\right\vert \,\int_{\{\varphi_{E_{k}}\leq M\}}u\,d\mu\\
& +\mathrm{Lip}(\psi)\,\gamma_{T}(\mu(E_{k}))\,M\text{.}%
\end{align*}
Since $\mathcal{U}$ is, in particular, uniformly integrable, there is some
$\delta>0$ such that
\begin{equation}
2\sup\left\vert \psi\right\vert \,\int_{A}u\,d\mu<\frac{\varepsilon}{6}\text{
\quad for all }u\in\mathcal{U}\text{ and }A\in\mathcal{A}\text{ with }%
\mu(A)<\delta\text{.}%
\end{equation}
But as $\mu(\varphi_{E_{k}}\leq M)\leq M\mu(E_{k})\rightarrow0$, this shows
there is some $k_{0}$ such that
\begin{equation}
\left\vert \int(\psi\circ R_{k})\,\left(  u-\frac{1}{M}%
{\textstyle\sum_{j=0}^{M-1}}
\widehat{T}^{j}u\right)  \,d\mu\right\vert <\frac{\varepsilon}{3}\text{ \quad
for }k\geq k_{0}\text{ and }u\in\mathcal{U}\text{.}\label{Eq_Smear2}%
\end{equation}
But combining (\ref{Eq_Smear}) with two applications of (\ref{Eq_Smear2})
yields%
\[
\left\vert \int(\psi\circ R_{k})\,u\,d\mu-\int(\psi\circ R_{k})\,u^{\ast
}\,d\mu\right\vert <\varepsilon\text{ \quad for }k\geq k_{0}\text{ and
}u,u^{\ast}\in\mathcal{U}\text{,}%
\]
which proves our earlier claim (\ref{Eq_TheUnifCgeTrick}).
\end{proof}%

\vspace{0.2cm}%

\section{The limit variables $\widetilde{\mathcal{H}}_{\alpha,\theta}$ and
$\mathcal{H}_{\alpha,\theta}$\label{Sec_ThePairs}}%

\noindent
\textbf{Further natural limit laws.} We continue our discussion of rare events
in the setup of the preceding positive results: Let $T$ be a c.e.m.p.t. on
$(X,\mathcal{A},\mu)$, $\mu(X)=\infty$, $Y\in\mathcal{A}$ a uniform set and
$a_{T}\in\mathcal{R}_{\alpha}$ for some $\alpha\in(0,1] $. Consider
$\gamma_{T}(\mu(E_{k}))\,\varphi_{E_{k}}$ for sequences $(E_{k})$ of
asymptotically rare events in $Y\cap\mathcal{A}$.

Among all possible limits for return- and hitting time distributions of
sequences in $Y$, the variables $\mathcal{H}_{\alpha}$, $\alpha\in(0,1]$,
stand out as the only limits which can occur simultaneously as asymptotic
hitting distribution and as asymptotic return distribution. This property
leads to the strategy for proving convergence to these particular laws
developed above, and via Theorem \ref{T_SuffForCgeToHAlpha} or its predecessor
in \cite{PSZ2} one finds that the $\mathcal{H}_{\alpha}$ occur at almost every
point of prototypical examples.

In the present section we discuss a larger family $(\widetilde{H}%
_{\alpha,\theta},H_{\alpha,\theta})$ with $\alpha\in(0,1],\theta\in
\lbrack0,1]$ of pairs $(\widetilde{F},F)$ in $\mathcal{F}$ related as in
(\ref{Eq_MagicFormula1}), which still appear in a very natural way. Let
$\mathcal{E}$, $\mathcal{G}_{\alpha}$ and $\Theta_{\theta}$ be independent
random variables, with $\mathcal{E}$ and $\mathcal{G}_{\alpha}$ as before, and
$\Pr[\Theta_{\theta}=1]=1-\Pr[\Theta_{\theta}=0]=\theta$. Define random
variables
\begin{equation}
\mathcal{H}_{\alpha,\theta}:=\theta^{-1/\alpha}\mathcal{H}_{\alpha}\text{
\quad and \quad}\widetilde{\mathcal{H}}_{\alpha,\theta}:=\Theta_{\theta}%
\cdot\theta^{-1/\alpha}\mathcal{H}_{\alpha}\text{,}\label{Eq_MoreRVs}%
\end{equation}
($0^{-1/\alpha}:=\infty$) with distribution functions $H_{\alpha,\theta}\ $and
$\widetilde{H}_{\alpha,\theta}$ respectively, given by
\begin{equation}
H_{\alpha,\theta}(t):=H_{\alpha}(\theta^{1/\alpha}t)\text{ \quad and \quad
}\widetilde{H}_{\alpha,\theta}(t):=(1-\theta)+\theta\,H_{\alpha}%
(\theta^{1/\alpha}t)\text{,\quad}t\geq0\text{.}\label{Eq_MoreDFs}%
\end{equation}
Obviously, $(\widetilde{H}_{\alpha,1},H_{\alpha,1})=(H_{\alpha},H_{\alpha})$,
while $(\widetilde{H}_{\alpha,0},H_{\alpha,0})=(1,0)$.

\begin{lemma}
[\textbf{Characterization of} \textbf{\ }$\widetilde{H}_{\alpha,\theta}$
\textbf{and }$H_{\alpha,\theta}$]\label{L_BasicsHAlphaTheta}For $\alpha
\in(0,1],\theta\in\lbrack0,1]$ the pair $(\widetilde{F},F):=(\widetilde
{H}_{\alpha,\theta},H_{\alpha,\theta})$ satisfies (\ref{Eq_MagicFormula1}).
The distribution function $\widetilde{F}=\widetilde{H}_{\alpha,\theta}$ is the
unique element of $\mathcal{F}$ which satisfies
\begin{equation}
\widetilde{F}(t)=(1-\theta)+\theta\,\int_{0}^{t}[1-\widetilde{F}%
(s)]\,\alpha\left(  t-s\right)  ^{\alpha-1}ds\text{ \quad for }t\geq
0\text{.}\label{Eq_MagicFormulaWithTheta}%
\end{equation}

\end{lemma}

\begin{proof}
Straightforward from the corresponding properties of $H_{\alpha}$, Lemma
\ref{L_CharHalpha}.
\end{proof}%

\vspace{0.2cm}%
%

\noindent
\textbf{Convergence to }$(\widetilde{H}_{\alpha,\theta},H_{\alpha,\theta})$
\textbf{in uniform sets.} This characterization of $\widetilde{H}%
_{\alpha,\theta}$ leads to an easy characterization (generalizing Theorem
\ref{T_Cge_to_H_alpha}) of those situations in which the pair $(\widetilde
{H}_{\alpha,\theta},H_{\alpha,\theta})$ occurs in the limit.

\begin{theorem}
[\textbf{Characterizing convergence to }$(\widetilde{H}_{\alpha,\theta
},H_{\alpha,\theta})$ \textbf{in uniform sets}]\label{T_Cge_to_H_alphatheta}%
Let $T$ be a c.e.m.p.t. on $(X,\mathcal{A},\mu) $, $\mu(X)=\infty$. Assume
that $Y\in\mathcal{A}$ is a uniform set and that $a_{T}\in\mathcal{R}_{\alpha
}$ for some $\alpha\in(0,1]$. Suppose that $\theta\in\lbrack0,1]$, and that
$E_{k}\subseteq Y$, $k\geq1$, are sets of positive measure with $\mu
(E_{k})\rightarrow0$.

Then the normalized return-time distributions of the $E_{k}$ converge to
$\widetilde{H}_{\alpha,\theta}$,
\begin{equation}
\mu_{E_{k}}(\gamma_{T}(\mu(E_{k}))\,\varphi_{E_{k}}\leq t)\Longrightarrow
\widetilde{H}_{\alpha,\theta}(t)\text{ \quad as }k\rightarrow\infty
\text{,}\label{Eq_Ret_Cge_H_alphaT}%
\end{equation}
if and only if the normalized hitting-time distributions converge to
$H_{\alpha,\theta}$,
\begin{equation}
\mu_{Y}(\gamma_{T}(\mu(E_{k}))\,\varphi_{E_{k}}\leq t)\Longrightarrow
H_{\alpha,\theta}(t)\text{ \quad as }k\rightarrow\infty\text{,}%
\label{Eq_Hit_Cge_H_alphaT}%
\end{equation}
if and only if for a dense set of points $t\ $in $(0,\infty)$,
\begin{equation}
\mu_{E_{k}}(\gamma_{T}(\mu(E_{k}))\,\varphi_{E_{k}}\leq t)-\theta\mu
_{Y}(\gamma_{T}(\mu(E_{k}))\,\varphi_{E_{k}}\leq t)\longrightarrow
1-\theta\text{ \quad as }k\rightarrow\infty\text{.}\label{Eq_ExtendedVersion}%
\end{equation}

\end{theorem}%

\vspace{0.2cm}%

\begin{proof}
By Theorem \ref{T_Ret_versus_Hit} and the first assertion of Lemma
\ref{L_BasicsHAlphaTheta}, (\ref{Eq_Ret_Cge_H_alphaT}) is equivalent to
(\ref{Eq_Hit_Cge_H_alphaT}). Via (\ref{Eq_MagicFormula1}) and
(\ref{Eq_MagicFormulaWithTheta}) either of these statements therefore implies
(\ref{Eq_ExtendedVersion}).

For the converse, start from (\ref{Eq_ExtendedVersion}), and assume for a
contradiction that, say, (\ref{Eq_Ret_Cge_H_alphaT}) fails, so that by Helly's
selection principle there is a subsequence $k_{j}\nearrow\infty$ of indices
and some $\widetilde{F}\in\mathcal{F}$, $\widetilde{F}\neq\widetilde
{H}_{\alpha,\theta}$, such that (\ref{Eq_AbstrDistrCgeII}) holds along that
subsequence. By Theorem \ref{T_Ret_versus_Hit}, so does
(\ref{Eq_AbstrDistrCgeIII}), where $F$ and $\widetilde{F}$ are related by
(\ref{Eq_MagicFormula1}). But then (\ref{Eq_ExtendedVersion}) ensures that
$\widetilde{F}=\widetilde{H}_{\alpha,\theta}$, see Lemma
\ref{L_BasicsHAlphaTheta}.

Exactly the same argument works if we assume that (\ref{Eq_Hit_Cge_H_alphaT}) fails.
\end{proof}%

\vspace{0.2cm}%

The pair $(\widetilde{H}_{\alpha,\theta},H_{\alpha,\theta})$ occurs in
situations where a proportion $1-\theta$ of $E_{k}$ returns very quickly (at a
rate smaller than $\gamma_{T}(\mu(E_{k}))$) to this very set, while the
remaining part of relative measure $\theta$ does not, and instead becomes
spread out macroscopically, as in the hitting-time statistics. This is made
precise in the following result which extends Theorem
\ref{T_SuffForCgeToHAlpha} to situations with $\theta\neq1$. In concrete maps,
this is what happens at (hyperbolic) periodic points, see Section
\ref{Sec_Exples} below.

\begin{theorem}
[\textbf{Sufficient conditions for convergence to }$(\widetilde{H}%
_{\alpha,\theta},H_{\alpha,\theta})$]\label{T_SuffForCgeToHAlphaTheta}Let $T$
be a c.e.m.p.t. on $(X,\mathcal{A},\mu)$, $\mu(X)=\infty$, pointwise dual
ergodic with $a_{T}\in\mathcal{R}_{\alpha}$ for some $\alpha\in(0,1]$. Suppose
that $Y$ is a uniform set, and that $E_{k}\subseteq Y$, $k\geq1$, are sets of
positive measure with $\mu(E_{k})\rightarrow0$.

Suppose that there are measurable functions $\tau_{k}:E_{k}\rightarrow
\mathbb{N}_{0}$ for which
\begin{equation}
\gamma_{T}(\mu(E_{k}))\,\tau_{k}\overset{\mu_{E_{k}}}{\longrightarrow}0\text{
\quad as }k\rightarrow\infty\text{.}\label{Eq_SmallDelay2}%
\end{equation}
Assume further that there is some $\theta\in\lbrack0,1]$ such that for each
$k\geq1$,
\begin{equation}
E_{k}=E_{k}^{\bullet}\cup E_{k}^{\circ}\text{\ (disjoint)\quad with\quad}%
\mu_{E_{k}}(E_{k}^{\circ})\longrightarrow\theta\text{ as }k\rightarrow
\infty\text{,}\label{Eq_SplittingEk}%
\end{equation}
and
\begin{equation}
\mu_{E_{k}^{\bullet}}(\varphi_{E_{k}}>\tau_{k})\longrightarrow0\text{ \quad as
}k\rightarrow\infty\text{,}\label{Eq_FinishedBeforeTau}%
\end{equation}
while
\begin{equation}
\mu_{E_{k}^{\circ}}(\varphi_{E_{k}}<\tau_{k})\longrightarrow0\text{ \quad as
}k\rightarrow\infty\text{,}\label{Eq_BoringBeforeTau2}%
\end{equation}
and there is some $L_{1}(\mu)$-compact set $\mathcal{U}\subseteq
\mathcal{D}(\mu)$ such that
\begin{equation}
\widehat{T^{\tau_{k}}}(\mu(E_{k}^{\circ})^{-1}1_{E_{k}^{\circ}})\in
\mathcal{U}\text{ \quad for }k\geq1\text{.}\label{Eq_TheCompactnessAssm2}%
\end{equation}
Then the return-time distributions of the $E_{k}$ converge to $\widetilde
{H}_{\alpha,\theta}$,
\begin{equation}
\mu_{E_{k}}(\gamma_{T}(\mu(E_{k}))\,\varphi_{E_{k}}\leq t)\longrightarrow
\widetilde{H}_{\alpha,\theta}(t)\text{ \quad as }k\rightarrow\infty
\text{,}\label{Eq_nvhfndhdhdnfhdhdhdhdhd}%
\end{equation}
and the hitting-time distributions converge to $H_{\alpha,\theta}$,
\begin{equation}
\mu_{Y}(\gamma_{T}(\mu(E_{k}))\,\varphi_{E_{k}}\leq t)\longrightarrow
H_{\alpha,\theta}(t)\text{ \quad as }k\rightarrow\infty\text{.}%
\label{Eq_nvhfndhdhdnfhdhdhdhdhd2}%
\end{equation}

\end{theorem}%

\vspace{0.2cm}%

\begin{proof}
We use Theorem \ref{T_Cge_to_H_alphatheta}. Represent $\widetilde{F}%
_{k}(t):=\mu_{E_{k}}(\gamma_{T}(\mu(E_{k}))\,\varphi_{E_{k}}\leq t)$ as
\[
\widetilde{F}_{k}(t)=\mu_{E_{k}}(E_{k}^{\bullet})\,\widetilde{F}_{k}^{\bullet
}(t)+\mu_{E_{k}}(E_{k}^{\circ})\,\widetilde{F}_{k}^{\circ}(t)
\]
with $\widetilde{F}_{k}^{\bullet}(t):=\mu_{E_{k}^{\bullet}}(\gamma_{T}%
(\mu(E_{k}))\,\varphi_{E_{k}}\leq t)$ and $\widetilde{F}_{k}^{\circ}%
(t):=\mu_{E_{k}^{\circ}}(\gamma_{T}(\mu(E_{k}))\,\varphi_{E_{k}}\leq t) $. By
assumptions (\ref{Eq_SplittingEk}), (\ref{Eq_FinishedBeforeTau}), and
(\ref{Eq_SmallDelay2}),
\[
\mu_{E_{k}}(E_{k}^{\bullet})\longrightarrow1-\theta\text{ \quad and \quad
}\widetilde{F}_{k}^{\bullet}(t)\longrightarrow1\text{ \quad as }%
k\rightarrow\infty\text{.}%
\]
To validate (\ref{Eq_ExtendedVersion}) it therefore suffices to show that for
a dense set of points $t\ $in $(0,\infty)$, the normalized hitting-time laws
$F_{k}(t):=\mu_{Y}(\gamma_{T}(\mu(E_{k}))\,\varphi_{E_{k}}\leq t)$ satisfy
\begin{equation}
\,\widetilde{F}_{k}^{\circ}(t)-F_{k}(t)\longrightarrow0\text{ \quad as
}k\rightarrow\infty\text{.}\label{Eq_NuZagn}%
\end{equation}
But the proof of (\ref{Eq_NuZagn}) works exactly like that of Theorem
\ref{T_SuffForCgeToHAlpha}: Simply replace $\,\widetilde{F}_{k}$ by
$\,\widetilde{F}_{k}^{\circ}$ and $v_{k}$ by $v_{k}^{\circ}:=\widehat
{T^{\tau_{k}}}(\mu(E_{k}^{\circ})^{-1}1_{E_{k}^{\circ}})$ etc.
\end{proof}%

\vspace{0.2cm}%

\begin{remark}
The $\theta=1$ case is a mild generalization of Theorem
\ref{T_SuffForCgeToHAlpha}. In the $\theta=0$ case it is clear that
assumptions (\ref{Eq_SmallDelay2})-(\ref{Eq_FinishedBeforeTau}) alone imply
(\ref{Eq_nvhfndhdhdnfhdhdhdhdhd}) and (\ref{Eq_nvhfndhdhdnfhdhdhdhdhd2}).
\end{remark}%

\vspace{0.2cm}%

\section{The barely recurrent case}%

\noindent
\textbf{The natural limit distributions.} The strategy of the preceding two
sections can also be used to obtain analogous limit theorems in the
\emph{barely recurrent} $\alpha=0$ case, where it is natural to study the
variables $\mu(E_{k})\,a_{T}(\varphi_{E_{k}})$ as discussed at the end of
Section 4.

We have already mentioned the limit law with distribution function $H^{\ast
}(t):=t/(1+t)$, $t\geq0$, obtained, for specific skew-product systems, in
\cite{PS}, \cite{PS2}, and \cite{PSZ}. Theorem \ref{T_Ret_versus_Hit2} now
allows us to show, via Remark \ref{Rem_Comments2} a), that for barely
recurrent general pointwise dual ergodic systems, $H^{\ast}$ plays the same
role as $H_{\alpha}$ did in the case $\alpha\in(0,1]$. More generally, there
are pairs $(\widetilde{H}_{\theta}^{\ast},H_{\theta}^{\ast})$, $\theta
\in\lbrack0,1]$, which play a role parallel to that of $(\widetilde{H}%
_{\alpha,\theta},H_{\alpha,\theta})$ in Section \ref{Sec_ThePairs}. Define
\begin{equation}
H_{\theta}^{\ast}(t):=H^{\ast}(\theta t)\text{\quad and\quad}\widetilde
{H}_{\theta}^{\ast}(t):=(1-\theta)+\theta H^{\ast}(\theta t)\text{,\quad}%
t\geq0\text{.}%
\end{equation}
Then $(\widetilde{H}_{1}^{\ast},H_{1}^{\ast})=(H^{\ast},H^{\ast})$, while
$(\widetilde{H}_{0}^{\ast},H_{0}^{\ast})=(1,0)$, and it is straightforward to
verify the following observation.

\begin{lemma}
[\textbf{Characterization of }$\widetilde{H}_{\theta}^{\ast}$\textbf{\ and
}$H_{\theta}^{\ast}$]\label{L_CharStars}For $\theta\in\lbrack0,1]$, the pair
$(\widetilde{G},G):=(\widetilde{H}_{\theta}^{\ast},H_{\theta}^{\ast})$
satisfies $G(t)=t[1-\widetilde{G}(t)]$ for $t\geq0$. The distribution function
$\widetilde{G}=\widetilde{H}_{\theta}^{\ast}$ is the unique element of
$\mathcal{F}$ which satisfies
\begin{equation}
\widetilde{G}(t)=(1-\theta)+\theta t[1-\widetilde{G}(t)]\text{\quad for }%
t\geq0\text{.}%
\end{equation}

\end{lemma}

We then obtain a result analogous to Theorems \ref{T_Cge_to_H_alpha} and
\ref{L_BasicsHAlphaTheta}.

\begin{theorem}
[\textbf{Characterizing convergence to }$(\widetilde{H}_{\theta}^{\ast
},H_{\theta}^{\ast})$ \textbf{in uniform sets}]%
\label{T_Cge_to_H_alphathetaStar}Let $T$ be a c.e.m.p.t. on $(X,\mathcal{A}%
,\mu)$, $\mu(X)=\infty$. Assume that $Y\in\mathcal{A}$ is a uniform set and
that $a_{T}\in\mathcal{R}_{0}$. Suppose that $\theta\in\lbrack0,1]$, and that
$E_{k}\subseteq Y$, $k\geq1$, are sets of positive measure with $\mu
(E_{k})\rightarrow0$.

Then the distorted return-time distributions of the $E_{k}$ converge to
$\widetilde{H}_{\theta}^{\ast}$,
\begin{equation}
\mu_{E_{k}}(\mu(E_{k})\,a_{T}(\varphi_{E_{k}})\leq t)\Longrightarrow
\widetilde{H}_{\theta}^{\ast}(t)\text{ \quad as }k\rightarrow\infty\text{,}%
\end{equation}
if and only if the distorted hitting-time distributions converge to
$H_{\theta}^{\ast}$,
\begin{equation}
\mu_{Y}(\mu(E_{k})\,a_{T}(\varphi_{E_{k}})\leq t)\Longrightarrow H_{\theta
}^{\ast}(t)\text{ \quad as }k\rightarrow\infty\text{,}%
\end{equation}
if and only if for a dense set of points $t\ $in $(0,\infty)$,
\begin{equation}
\mu_{E_{k}}(\mu(E_{k})\,a_{T}(\varphi_{E_{k}})\leq t)-\theta\mu_{Y}(\mu
(E_{k})\,a_{T}(\varphi_{E_{k}})\leq t)\longrightarrow1-\theta\text{ \quad as
}k\rightarrow\infty\text{.}\label{Eq_ExtendedVersion2}%
\end{equation}

\end{theorem}%

\vspace{0.2cm}%

\begin{proof}
Like Theorems \ref{T_Cge_to_H_alpha} and \ref{L_BasicsHAlphaTheta}, using
Theorem \ref{T_Ret_versus_Hit2} and Lemma \ref{L_CharStars}.
\end{proof}%

\vspace{0.2cm}%

This, in turn, leads to sufficient conditions for convergence to $H^{\ast}$
and, more generally, to $(\widetilde{H}_{\theta}^{\ast},H_{\theta}^{\ast})$,
parallel to those of Theorems \ref{T_SuffForCgeToHAlpha} and
\ref{T_SuffForCgeToHAlphaTheta} above.

\begin{theorem}
[\textbf{Sufficient conditions for convergence to }$(\widetilde{H}_{\theta
}^{\ast},H_{\theta}^{\ast})$]\label{T_SuffForCgeToHAlphaThetaStar}Let $T$ be a
c.e.m.p.t. on $(X,\mathcal{A},\mu)$, $\mu(X)=\infty$, pointwise dual ergodic
with $a_{T}\in\mathcal{R}_{0}$. Suppose that $Y$ is a uniform set, and that
$E_{k}\subseteq Y$, $k\geq1$, are sets of positive measure with $\mu
(E_{k})\rightarrow0$.

Suppose that there are measurable functions $\tau_{k}:E_{k}\rightarrow
\mathbb{N}_{0}$ for which
\begin{equation}
\mu(E_{k})\,a_{T}(\tau_{k})\overset{\mu_{E_{k}}}{\longrightarrow}0\text{ \quad
as }k\rightarrow\infty\text{.}\label{Eq_Schandfleck}%
\end{equation}
Assume further that there is some $\theta\in\lbrack0,1]$ such that for each
$k\geq1$,
\begin{equation}
E_{k}=E_{k}^{\bullet}\cup E_{k}^{\circ}\text{\ (disjoint)\quad with\quad}%
\mu_{E_{k}}(E_{k}^{\circ})\longrightarrow\theta\text{ as }k\rightarrow
\infty\text{,}\label{Eq_bhvdhjbhjh}%
\end{equation}
and
\begin{equation}
\mu_{E_{k}^{\bullet}}(\varphi_{E_{k}}>\tau_{k})\longrightarrow0\text{ \quad as
}k\rightarrow\infty\text{,}\label{Eq_khgkhgkhjkhgkf}%
\end{equation}
while
\begin{equation}
\mu_{E_{k}^{\circ}}(\varphi_{E_{k}}<\tau_{k})\longrightarrow0\text{ \quad as
}k\rightarrow\infty\text{,}\label{Eq_hdfhsdhshsahsadhsdhsahsah}%
\end{equation}
and there is some $L_{1}(\mu)$-compact set $\mathcal{U}\subseteq
\mathcal{D}(\mu)$ such that
\begin{equation}
\widehat{T^{\tau_{k}}}(\mu(E_{k}^{\circ})^{-1}1_{E_{k}^{\circ}})\in
\mathcal{U}\text{ \quad for }k\geq1\text{.}\label{Eq_cbcbcbcbcbc}%
\end{equation}
Then the distorted return-time distributions of the $E_{k}$ converge to
$\widetilde{H}_{\theta}^{\ast}$,
\begin{equation}
\mu_{E_{k}}(\mu(E_{k})\,a_{T}(\varphi_{E_{k}})\leq t)\longrightarrow
\widetilde{H}_{\theta}^{\ast}(t)\text{ \quad as }k\rightarrow\infty
\text{,}\label{Eq_csdvgcsdvgsd}%
\end{equation}
and the distorted hitting-time distributions converge to $H_{\theta}^{\ast}
$,
\begin{equation}
\mu_{Y}(\mu(E_{k})\,a_{T}(\varphi_{E_{k}})\leq t)\longrightarrow H_{\theta
}^{\ast}(t)\text{ \quad as }k\rightarrow\infty\text{.}\label{Eq_csdvgcsdvgsd2}%
\end{equation}

\end{theorem}%

\vspace{0.2cm}%

\begin{proof}
\textbf{(i)} We will employ Theorem \ref{T_Cge_to_H_alphathetaStar}. Let
$\mathsf{R}_{k}:=\mu(E_{k})\,a_{T}(\varphi_{E_{k}})$ and represent
$\widetilde{G}_{k}(t):=\mu_{E_{k}}(\mathsf{R}_{k}\leq t)$ as
\[
\widetilde{G}_{k}(t)=\mu_{E_{k}}(E_{k}^{\bullet})\,\widetilde{G}_{k}^{\bullet
}(t)+\mu_{E_{k}}(E_{k}^{\circ})\,\widetilde{G}_{k}^{\circ}(t)
\]
with $\widetilde{G}_{k}^{\bullet}(t):=\mu_{E_{k}^{\bullet}}(\mathsf{R}_{k}\leq
t)$ and $\widetilde{G}_{k}^{\circ}(t):=\mu_{E_{k}^{\circ}}(\mathsf{R}_{k}\leq
t)$. By (\ref{Eq_bhvdhjbhjh}), (\ref{Eq_khgkhgkhjkhgkf}) and
(\ref{Eq_Schandfleck}),
\[
\mu_{E_{k}}(E_{k}^{\bullet})\longrightarrow1-\theta\text{ \quad and \quad
}\widetilde{G}_{k}^{\bullet}(t)\longrightarrow1\text{ \quad as }%
k\rightarrow\infty\text{.}%
\]
To validate (\ref{Eq_ExtendedVersion2}) it therefore suffices to show that for
a dense set of points $t\ $in $(0,\infty)$, the distorted hitting-time laws
$G_{k}(t):=\mu_{Y}(\mathsf{R}_{k}\leq t)$ satisfy
\begin{equation}
\,\widetilde{G}_{k}^{\circ}(t)-G_{k}(t)\longrightarrow0\text{ \quad as
}k\rightarrow\infty\text{.}%
\end{equation}
We prove the latter by showing that
\begin{equation}
\widetilde{G}_{k}^{\circ}-\widehat{G}_{k}^{\circ}\rightarrow0\text{\quad
and\quad}\widehat{G}_{k}^{\circ}-G_{k}\rightarrow0\text{ \quad}\lambda
\text{-a.e. on }[0,\infty)\text{ as }k\rightarrow\infty\text{,}%
\end{equation}
where $\widehat{G}_{k}$ denotes the distribution function of $\mathsf{R}_{k}$
with respect to the measure with density $v_{k}^{\circ}:=\widehat{T^{\tau_{k}%
}}(\mu(E_{k}^{\circ})^{-1}1_{E_{k}^{\circ}})$, so that $\widehat{G}_{k}%
^{\circ}(t):=\int_{\{\mathsf{R}_{k}\leq t\}}v_{k}^{\circ}\,d\mu$.%
\newline
%

\noindent
\textbf{(ii)} Recall that by definition (\ref{Eq_vcycvyvcvycvshgdv}) of
$\widehat{T^{\tau}}$ we have $\widehat{G}_{k}^{\circ}(t)=\mu_{E_{k}^{\circ}%
}(\mathsf{R}_{k}\circ T^{\tau_{k}}\leq t)$. To first check that for $\lambda
$-a.e. $t\in\lbrack0,\infty)$,
\begin{equation}
\mu_{E_{k}^{\circ}}(\mathsf{R}_{k}\leq t)-\mu_{E_{k}^{\circ}}(\mathsf{R}%
_{k}\circ T^{\tau_{k}}\leq t)\longrightarrow0\text{ \quad as }k\rightarrow
\infty\text{.}\label{Eq_lhjkkhjgkjgklgkjjgngklmjk2}%
\end{equation}
recall (\ref{Eq_xvcygvdshgvchasvaaaaaa}) and (\ref{Eq_AnSubadd}) to see that
\begin{align}
\mathsf{R}_{k}-\mathsf{R}_{k}\circ T^{\tau_{k}}  & =\mu(E_{k})\left(
a_{T}(\varphi_{E_{k}}\circ T^{\tau_{k}}+\tau_{k})-a_{T}(\varphi_{E_{k}}\circ
T^{\tau_{k}})\right) \nonumber\\
& \leq\mu(E_{k})\,a_{T}(\tau_{k})\text{ \quad on }\{\varphi_{E_{k}}>\tau
_{k}\}\text{.}%
\end{align}
Now (\ref{Eq_Schandfleck}) and (\ref{Eq_hdfhsdhshsahsadhsdhsahsah}) imply that
$\widetilde{G}_{k}^{\circ}-\widehat{G}_{k}^{\circ}\rightarrow0$
holds\ $\lambda$-a.e. on $[0,\infty)$.%
\newline
%

\noindent
\textbf{(iii)} It remains to prove that $\widehat{G}_{k}^{\circ}%
-G_{k}\rightarrow0$ a.e. on $[0,\infty)$, which is equivalent to
\begin{equation}
\int(\psi\circ\mathsf{R}_{k})\,1_{Y}\,d\mu-\int(\psi\circ\mathsf{R}%
_{k})\,v_{k}^{\circ}\,d\mu\longrightarrow0\text{ \quad as }k\rightarrow\infty
\end{equation}
for every bounded Lipschitz function $\psi:\mathbb{R}\rightarrow\mathbb{R}$.
In view of (\ref{Eq_cbcbcbcbcbc}), this follows if we show that for any such
$\psi$,
\begin{equation}
\int(\psi\circ\mathsf{R}_{k})\,u\,d\mu-\int(\psi\circ\mathsf{R}_{k})\,u^{\ast
}\,d\mu\longrightarrow0\text{ \quad}%
\begin{array}
[c]{c}%
\text{as }k\rightarrow\infty\text{,}\\
\text{uniformly in }u,u^{\ast}\in\mathcal{U}%
\end{array}
\label{Eq_TheUnifCgeTrick1}%
\end{equation}
(assume w.l.o.g. that $1_{Y}\in\mathcal{U}$). To this end, we now fix $\psi$
and any $\varepsilon>0$.%
\newline
%

\noindent
\textbf{(iv)} As a consequence of Remark \ref{R_UnifMETH} there is some
$M\geq1$ such that
\begin{align}
\left\vert \int(\psi\circ\mathsf{R}_{k})\,\frac{1}{M}%
{\textstyle\sum_{j=0}^{M-1}}
\widehat{T}^{j}(u-u^{\ast})\,d\mu\right\vert  & \leq\sup\left\vert
\psi\right\vert \,\left\Vert \frac{1}{M}%
{\textstyle\sum_{j=0}^{M-1}}
\widehat{T}^{j}(u-u^{\ast})\right\Vert _{1}\nonumber\\
& <\frac{\varepsilon}{3}\text{ \quad for all }u,u^{\ast}\in\mathcal{U}%
\text{.}\label{Eq_hgdvchsagvchasvchasdvcasvcsavcsacd}%
\end{align}
For any $j\in\{0,\ldots,M-1\}$ and $u\in\mathcal{U}$,
(\ref{Eq_xvcygvdshgvchasvaaaaaa}) and (\ref{Eq_AnSubadd}) entail that
\begin{align*}
\mathsf{R}_{k}-\mathsf{R}_{k}\circ T^{j}  & =\mu(E_{k})\left(  a_{T}%
(\varphi_{E_{k}}\circ T^{j}+j)-a_{T}(\varphi_{E_{k}}\circ T^{j})\right) \\
& \leq\mu(E_{k})\,a_{T}(j)\text{ \quad on }\{\varphi_{E_{k}}>j\}\text{,}%
\end{align*}
and therefore%
\[
\left\vert \int(\psi\circ\mathsf{R}_{k})\,(u-\widehat{T}^{j}u)\,d\mu
\right\vert \leq2\sup\left\vert \psi\right\vert \,\int_{\{\varphi_{E_{k}}\leq
j\}}u\,d\mu+\mathrm{Lip}(\psi)\,a_{T}(j)\,\mu(E_{k})\text{,}%
\]
so that
\begin{align*}
\left\vert \int(\psi\circ\mathsf{R}_{k})\,\left(  u-\frac{1}{M}%
{\textstyle\sum_{j=0}^{M-1}}
\widehat{T}^{j}u\right)  \,d\mu\right\vert  & \leq2\sup\left\vert
\psi\right\vert \,\int_{\{\varphi_{E_{k}}\leq M\}}u\,d\mu\\
& +\mathrm{Lip}(\psi)\,a_{T}(M)\,\mu(E_{k})\text{.}%
\end{align*}
We then find, exactly as in the proof of Theorem \ref{T_Cge_to_H_alpha}, some
$k_{0}\geq1$ such that
\[
\left\vert \int(\psi\circ\mathsf{R}_{k})\,u\,d\mu-\int(\psi\circ\mathsf{R}%
_{k})\,u^{\ast}\,d\mu\right\vert <\varepsilon\text{ \quad for }k\geq
k_{0}\text{ and }u,u^{\ast}\in\mathcal{U}\text{,}%
\]
which proves our earlier claim (\ref{Eq_TheUnifCgeTrick1}).
\end{proof}%

\vspace{0.2cm}%

\begin{remark}
In the $\theta=0$ case it is clear that assumptions (\ref{Eq_Schandfleck}%
)-(\ref{Eq_khgkhgkhjkhgkf}) alone imply (\ref{Eq_csdvgcsdvgsd}) and
(\ref{Eq_csdvgcsdvgsd2}).
\end{remark}

\begin{remark}
\label{R_cdvbgcdbcd}Note that the assumption $\alpha=0$ enters the argument
only when we appeal to Theorem \ref{T_Cge_to_H_alphathetaStar}. Indeed,
arguing as in our reduction of \ref{T_Ret_versus_Hit} to
\ref{T_Ret_versus_Hit2}, we could derive Theorems \ref{T_SuffForCgeToHAlpha}
and \ref{T_SuffForCgeToHAlphaTheta} using the above proof.
\end{remark}%

\vspace{0.2cm}%

\section{Application to basic prototypical systems\label{Sec_Exples}}%

\noindent
\textbf{A basic family of null-recurrent interval maps.} To illustrate the use
of the abstract Theorems \ref{T_SuffForCgeToHAlpha},
\ref{T_SuffForCgeToHAlphaTheta} and \ref{T_SuffForCgeToHAlphaThetaStar}, in
the simplest nontrivial setup, we consider a family of interval maps with an
indifferent fixed point and two full branches.

For standard classes of probability preserving maps, return- and hitting-time
statistics for neighborhoods of hyperbolic periodic points are well understood
(see e.g. \cite{K}, \cite{F}). Examples \ref{Ex_B} and \ref{Ex_D} below appear
to be the first results on return- and hitting-time statistics at hyperbolic
fixed points of smooth infinite measure preserving maps. (They extend to
hyperbolic fixed or periodic points of more general maps like those of
\cite{Z1}, \cite{Z2} by routine arguments.)

Examples \ref{Ex_C} and \ref{Ex_D}, on the other hand, seem to be the first to
identify nontrivial return- and hitting-time statistics for $\alpha=0$ maps
without skew product structure. \newline%

\noindent
Consider maps $T$ on $X=[0,1]$ which satisfy, for some $c\in(0,1)$, the following:

\begin{enumerate}
\item[(a)] The restrictions of $T$ to $(0,c)$ and $(c,1)$ map onto $(0,1)$,
and possess $\mathcal{C}^{2}$-extensions to $[0,c]$ and $[c,1]$ respectively.

\item[(b)] $T^{\prime}(0)=1$ while $T^{\prime}>1$ on $(0,c]\cup\lbrack c,1]$.

\item[(c)] $T^{\prime}$ is increasing in some neighbourhood of $0$.
\end{enumerate}%

\noindent
According to \cite{T2}, any such $T$ has an infinite invariant measure $\mu$
with a density $h$ relative to Lebesgue measure $\lambda$ which is strictly
positive and continuous on $(0,1]$, and $T$ is conservative ergodic for $\mu$.
Due to \cite{A2} (or \cite{T3}), $Y:=(c,1)$ is a DK-set.

For simplicity we look at a concrete family of such maps $T=T_{\alpha,\theta}$
with $c=c_{\alpha,\theta}=\theta$, parametrized by $\alpha\in\lbrack0,1]$ and
$\theta\in(0,1)$, for which $T$ is affine with slope $1/(1-\theta)$ on
$[\theta,1]$, while

\begin{enumerate}
\item[(d)] $Tx=\left\{
\begin{array}
[c]{cc}%
x+\vartheta_{\alpha,\theta}\,x^{1+1/\alpha} & \text{if }\alpha\leq1\text{,}\\
x+\vartheta_{\alpha,\theta}\,x^{2}e^{-1/x} & \text{if }\alpha=0\text{,}%
\end{array}
\right.  $ for $x\in(0,\theta]$, \newline
\end{enumerate}%

\noindent
with $\vartheta=\vartheta_{\alpha,\theta}>0$ the appropriate constant. Then
(again \cite{A2} or \cite{T3}), $a_{T}\in\mathcal{R}_{\alpha}$ where (for a
constant $\kappa_{\alpha,\theta}$, explicit in terms of the invariant density
$h=h_{\alpha,\theta}$)
\begin{equation}
a_{T}(s)\sim\kappa_{\alpha,\theta}\cdot\left\{
\begin{array}
[c]{cc}%
s/\log s & \text{if }\alpha=1\text{,}\\
s^{\alpha} & \text{if }\alpha<1\text{,}\\
\log s & \text{if }\alpha=0\text{,}%
\end{array}
\right.  \quad\text{as }s\rightarrow\infty\text{,}\label{Eq_AsyConcreteA}%
\end{equation}
and hence%
\begin{equation}
\gamma_{T}(s)\sim\kappa_{\alpha,\theta}^{\alpha}\cdot\left\{
\begin{array}
[c]{cc}%
s/(-\log s) & \text{if }\alpha=1\text{,}\\
s^{1/\alpha} & \text{if }\alpha<1\text{,}%
\end{array}
\right.  \quad\text{as }s\searrow0\text{.}%
\end{equation}%
\vspace{0.2cm}%
%

\noindent
\textbf{Concrete sequences of asymptotically rare events.} For these maps
there are simple sequences $(E_{k})$ for which all types of limit laws
discussed before appear:

Let $c=:b_{1}<b_{2}<\ldots<b_{k}\nearrow1$ be such that $Tb_{k+1}=b_{k}$, and
define $E_{k}^{\prime}:=(b_{k},b_{k+1})$, $E_{k}^{\prime\prime}:=(b_{k},1)$,
$k\geq1$. The $E_{k}^{\prime\prime}$ are cylinders shrinking to the hyperbolic
fixed point $x=0$, while the $E_{k}^{\prime}$ approach but never contain that
point. A calculation gives $\lambda(E_{k}^{\prime})=\theta(1-\theta)^{k}$,
$\lambda(E_{k}^{\prime\prime})=(1-\theta)^{k}$, while $\mu(E_{k}^{\prime})\sim
h(1)\lambda(E_{k}^{\prime})$ and $\mu(E_{k}^{\prime\prime})\sim h(1)\lambda
(E_{k}^{\prime\prime})$ as $k\rightarrow\infty$. Note that
\[
\varphi_{E_{k}^{\prime}}\geq\varphi_{E_{k}^{\prime\prime}}\geq k\text{ on
}E_{k}^{\prime}\text{ \quad and \quad}\varphi_{E_{k}^{\prime\prime}}=1\text{
on }E_{k}^{\prime\prime}\setminus E_{k}^{\prime}\text{.}%
\]
Let $\mathcal{D}_{r}$ be the collection of all probability densities on
$F:=(0,c)$ which have a version respecting $r$ as a Lipschitz constant. By
Arzela-Ascoli, each $\mathcal{D}_{r}$ is relatively compact in $L_{1}(\mu)$.
Setting $\tau_{k}:=k$, we have
\[
v_{k}:=\widehat{T^{\tau_{k}}}(\mu(E_{k}^{\prime})^{-1}1_{E_{k}^{\prime}}%
)=\mu(E_{k}^{\prime})(f^{k})^{\prime}\text{ \quad for }k\geq1\text{,}%
\]
where $f:=(T\mid_{Y})^{-1}$ is the inverse of the uniformly expanding branch
of $T$. By the standard bounded distortion estimate for uniformly expanding
$\mathcal{C}^{2}$-maps (e.g. \S 4.3 of \cite{A0}), we see that $v_{k}%
\in\mathcal{D}_{r}$ for all $k\geq1$, provided that $r\geq r(\alpha,\theta)$
for some explicitly known $r(\alpha,\theta)$. Using these observations, we
easily obtain the following specific limit theorems. (Recall that for
$\alpha\in(0,1)$ and $\theta\in(0,1]$ we don't have an explicit expression for
$H_{\alpha}(t)$.)

\begin{example}
[\textbf{Convergence to }$H_{\alpha}$ \textbf{with} $0<\alpha\leq 1 $]\label{Ex_A}
Let $T=T_{\alpha,\theta}$ with $\alpha\in(0,1]$, $\theta\in(0,1)$
and consider the sequence $(E_{k})_{k\geq1}$ with $E_{k}:=E_{k}^{\prime}$.
Then Theorem \ref{T_SuffForCgeToHAlpha} applies (with $\tau_{k}:=k$), showing,
for $\alpha\in(0,1)$, that%
\[%
\begin{array}
[c]{c}%
\mu_{E_{k}}(\kappa_{\alpha,\theta}^{\alpha}(h(1)\theta)^{1/\alpha}%
(1-\theta)^{k/\alpha}\,\varphi_{E_{k}}\leq t)\Longrightarrow H_{\alpha}(t)\\
\mu_{Y}(\kappa_{\alpha,\theta}^{\alpha}(h(1)\theta)^{1/\alpha}(1-\theta
)^{k/\alpha}\,\varphi_{E_{k}}\leq t)\Longrightarrow H_{\alpha}(t)
\end{array}
\text{ \quad as }k\rightarrow\infty\text{.}%
\]
In case $\alpha=1$ we get
\[%
\begin{array}
[c]{c}%
\mu_{E_{k}}\left(  \frac{\kappa_{1,\theta}h(1)\theta}{-\log(1-\theta)}%
\,\frac{(1-\theta)^{k}}{k}\,\varphi_{E_{k}}\leq t\right)  \Longrightarrow
1-e^{-t}\\
\mu_{Y}\left(  \frac{\kappa_{1,\theta}h(1)\theta}{-\log(1-\theta)}%
\,\frac{(1-\theta)^{k}}{k}\,\varphi_{E_{k}}\leq t\right)  \Longrightarrow
1-e^{-t}%
\end{array}
\text{ \quad as }k\rightarrow\infty\text{.}%
\]
(For this example we could also use Theorem 4.1 of \cite{PSZ2}, but only if we
appeal to an additional result, Proposition 6.2 of that paper.)
\end{example}

\begin{example}
[\textbf{Convergence to }$H_{\alpha,\theta}$ \textbf{with} $0<\alpha\leq 1$,
$\theta\in(0,1)$]\label{Ex_B}Let $T=T_{\alpha,\theta}$ with $\alpha\in(0,1]$,
$\theta\in(0,1)$ and consider the sequence $(E_{k})_{k\geq1}$ with
$E_{k}:=E_{k}^{\prime\prime}$. Then Theorem \ref{T_SuffForCgeToHAlphaTheta}
applies (with $\tau_{k}:=k$ and $E_{k}^{\circ}:=E_{k}^{\prime}$), and we find,
for $\alpha\in(0,1)$, that
\[%
\begin{array}
[c]{c}%
\mu_{E_{k}}(\kappa_{\alpha,\theta}^{\alpha}h(1)^{1/\alpha}(1-\theta
)^{k/\alpha}\,\varphi_{E_{k}}\leq t)\Longrightarrow(1-\theta)+\theta
\,H_{\alpha}(\theta^{1/\alpha}t)\\
\mu_{Y}(\kappa_{\alpha,\theta}^{\alpha}h(1)^{1/\alpha}(1-\theta)^{k/\alpha
}\,\varphi_{E_{k}}\leq t)\Longrightarrow H_{\alpha}(\theta^{1/\alpha}t)
\end{array}
\text{ as }k\rightarrow\infty\text{.}%
\]
In case $\alpha=1$ we get
\[%
\begin{array}
[c]{c}%
\mu_{E_{k}}\left(  \frac{\kappa_{1,\theta}h(1)}{-\log(1-\theta)}%
\,\frac{(1-\theta)^{k}}{k}\,\varphi_{E_{k}}\leq t\right)  \Longrightarrow
1-\theta e^{-\theta^{1/\alpha}t}\\
\mu_{Y}\left(  \frac{\kappa_{1,\theta}h(1)}{-\log(1-\theta)}\,\frac
{(1-\theta)^{k}}{k}\,\varphi_{E_{k}}\leq t\right)  \Longrightarrow
1-e^{-\theta^{1/\alpha}t}%
\end{array}
\text{ \quad as }k\rightarrow\infty\text{.}%
\]
(A much simpler (piecewise affine) map for which hitting-time distributions of
a hyperbolic fixed point converge to $H_{\alpha,\theta}$ has been discussed in
\cite{BZ}.)
\end{example}

\begin{example}
[\textbf{Convergence to }$H^{\ast}$ \textbf{with }$\alpha=0$]\label{Ex_C}Let
$T=T_{0,\theta}$ with $\theta\in(0,1)$ and consider the sequence
$(E_{k})_{k\geq1}$ with $E_{k}:=E_{k}^{\prime}$. Then Theorem
\ref{T_SuffForCgeToHAlphaThetaStar} applies (with $\tau_{k}:=k$ and
$E_{k}^{\circ}:=E_{k}^{\prime}$), showing that%
\[%
\begin{array}
[c]{c}%
\mu_{E_{k}}(\kappa_{0,\theta}h(1)\theta(1-\theta)^{k}\,\log(\varphi_{E_{k}%
})\leq t)\Longrightarrow\frac{t}{1+t}\\
\mu_{Y}(\kappa_{0,\theta}h(1)\theta(1-\theta)^{k}\,\log(\varphi_{E_{k}})\leq
t)\Longrightarrow\frac{t}{1+t}%
\end{array}
\text{ \quad as }k\rightarrow\infty\text{.}%
\]

\end{example}

\begin{example}
[\textbf{Convergence to }$H_{\theta}^{\ast}$ \textbf{with }$\alpha=0$,
$\theta\in(0,1)$]\label{Ex_D}Let $T=T_{0,\theta}$ with $\theta\in(0,1)$ and
consider the sequence $(E_{k})_{k\geq1}$ with $E_{k}:=E_{k}^{\prime\prime}$.
Then Theorem \ref{T_SuffForCgeToHAlphaThetaStar} applies (with $\tau_{k}:=k$
and $E_{k}^{\circ}:=E_{k}^{\prime}$), showing that%
\[%
\begin{array}
[c]{c}%
\mu_{E_{k}}(\kappa_{0,\theta}h(1)(1-\theta)^{k}\,\log(\varphi_{E_{k}})\leq
t)\Longrightarrow\frac{(1-\theta)+\theta t}{1+\theta t}\\
\mu_{Y}(\kappa_{0,\theta}h(1)(1-\theta)^{k}\,\log(\varphi_{E_{k}})\leq
t)\Longrightarrow\frac{\theta t}{1+\theta t}%
\end{array}
\text{ \quad as }k\rightarrow\infty\text{.}%
\]

\end{example}%

\vspace{0.2cm}%

\begin{remark}
We finally mention that hitting-time statistics for neighborhoods of a neutral
fixed point have been discussed in \cite{ZFundMath2}, which covers the above
maps $T=T_{\alpha,\theta}$ in case $\alpha\in(0,1]$, and applies to the
sequence $(F_{k})_{k\geq1}$ of cylinders shrinking to the point $x=0$ (which,
however, satisfy $\mu(F_{k})=\infty$ for all $k$).
\end{remark}%

\vspace{0.2cm}%

\end{document}